\documentclass[11pt,leqno]{amsart}

\usepackage{amsmath,amsthm,amsfonts,amssymb}

\usepackage{enumerate}
\usepackage{mathtools}
\usepackage{graphicx}
\usepackage{tikz}
\usepackage{tikz-cd}
\usepackage{float}
\usepackage{hyperref}
\usepackage{enumitem}
\usepackage{algorithm}
\usepackage{algpseudocode}
\usepackage{comment}
\usepackage[english]{babel}
\usepackage[margin=1.15in]{geometry}

\theoremstyle{plain}

\usepackage{aliascnt}
\usepackage{cleveref}
\newtheorem{theorem}{Theorem}[section]

\newaliascnt{proposition}{theorem}
\newtheorem{proposition}[proposition]{Proposition}
\aliascntresetthe{proposition}

\newaliascnt{corollary}{theorem}
\newtheorem{corollary}[corollary]{Corollary}
\aliascntresetthe{corollary}

\newaliascnt{lemma}{theorem}
\newtheorem{lemma}[lemma]{Lemma}
\aliascntresetthe{lemma}

\newaliascnt{definition}{theorem}
\newtheorem{definition}[definition]{Definition}
\aliascntresetthe{definition}

\newaliascnt{conjecture}{theorem}

\aliascntresetthe{conjecture}

\newaliascnt{question}{theorem}
\newtheorem{question}[question]{Question}
\aliascntresetthe{question}

\newaliascnt{example}{theorem}
\newtheorem{example}[example]{Example}
\aliascntresetthe{example}

\newaliascnt{claim}{theorem}

\aliascntresetthe{claim}

\newaliascnt{remark}{theorem}
\newtheorem{remark}[remark]{Remark}
\aliascntresetthe{remark}

\newaliascnt{problem}{theorem}

\aliascntresetthe{problem}

\newtheorem{thmx}{Theorem}

\newcommand\T{{\mathcal{T}}}

\DeclareMathOperator{\Aut}{Aut}
\DeclareMathOperator{\St}{St}

\DeclareMathOperator{\Rist}{Rist}

\newcommand{\alt}{\mathrm{Alt}}
\newcommand{\Sym}{\mathrm{Sym}}
\newcommand{\Syl}{\mathrm{syl}}

\newcommand{\Dfive}{\mathcal D_{5}}
\DeclareMathOperator{\st}{st}
\DeclareMathOperator{\rist}{rist}

\newcommand{\ml}[1]{\textcolor{magenta}{#1}}

\numberwithin{equation}{section}

\title[Spinal Hanoi Towers Groups]{Spinal Hanoi Towers Groups}

\author[F.~Cavalieri]{Francesca Cavalieri}
\address{Francesca Cavalieri: Università di Padova, Dipartimento di Matematica ``Tullio Levi-Civita", Via Trieste 63, 35121 Padova, Italy; Department of Mathematics, University of the Basque Country UPV/EHU, 48080 Bilbao, Spain}
\email{francesca.cavalieri.1@phd.unipd.it}

\author[M.\,E.~Garciarena]{Mikel E.\ Garciarena}
\address{Mikel E.\ Garciarena: Heinrich Heine University D\"usseldorf,
Faculty of Mathematics and Natural Sciences,
Mathematical Institute, Universit\"atsstr.\ 1, 40225 D\"usseldorf, Germany.}
\email{mikel.eguzki.garciarena.perez@hhu.de}

\author[M.~Noce]{Marialaura Noce}
\address{Marialaura Noce: Dipartimento di Informatica, Universit\`a di Salerno, 84084 Fisciano, Italy}
\email{mnoce@unisa.it}

\date{}



\begin{document}

\begin{abstract}
We introduce and study \emph{spinal Hanoi towers groups}, a family of groups acting on the $d$-adic tree that contains both the classical Hanoi towers group $\mathcal{H}^{(3)}$ and Skipper's generalizations as extreme cases. Each group is generated by $d$ automorphisms $a_1,\dots,a_d$, where $a_i$ has a unique non-trivial section, equal to $a_i$ itself, at the $i$-th coordinate, and root permutation $\sigma_i$ fixing $i$. The entire construction is thus encoded by the finite permutation group $P=\langle\sigma_1,\dots,\sigma_d\rangle\leq\mathrm{Sym}(d)$, and we develop a dictionary between the two: $G$ is fractal and level transitive if and only if $P$ is transitive; every group in the family is amenable and contracting, with explicit nucleus and a word problem solved by a length-halving recursion on syllables; and the abelianizations of $G$ and $P$ together control the first level stabilizer.

The branch structure of the family is governed by the subgroup $J\leq\Aut(\T_d)$, generated by the automorphisms with exactly two non-trivial sections, occupied by an element $h\in G$ and its inverse. We give a criterion for the containment $J\leq G$ that can be verified in an explicit finite quotient, and we prove that whenever it holds, a level transitive spinal Hanoi towers group is strongly fractal, regular branch over its commutator subgroup, has explicitly described rigid stabilizers at every level, and is just infinite. As an application, we show that the groups of type $(d,m)$, whose root permutations are $m$-cycles, satisfy $J\leq G$ for all $d\geq 4$ and are therefore just infinite, in contrast with the classical case of the Hanoi towers group on three pegs. 
\end{abstract}

\maketitle

%\tableofcontents

\section{Introduction}

Groups acting on regular rooted trees provide a rich source of examples and counterexamples in infinite group theory. Since the pioneering work of Grigorchuk \cite{Grigorchuk84}, who constructed the first finitely generated group of intermediate word growth, self-similar groups have played a central role in geometric and combinatorial group theory, as well as in dynamics, by virtue of their connections to iterated monodromy groups and limit dynamical systems \cite{BartholdiNekrashevych08, ssg}. Important classes of such groups include branch groups (see the survey by Bartholdi, Grigorchuk and \v{S}uni\'c \cite{BGS03} and Grigorchuk's foundational paper \cite{JI grigorchuk}), contracting groups, and groups defined by finite-state automata; we refer to the monograph of Nekrashevych \cite{ssg} for a comprehensive treatment.

Among the historically most prominent examples of self-similar groups, beyond the Grigorchuk group itself \cite{Grigorchuk80}, are the Gupta-Sidki $p$-groups \cite{GuptaSidki83}, providing the first finitely generated infinite $p$-groups for odd primes, and the Basilica group of Grigorchuk and \.{Z}uk \cite{GrigorchukZuk02}, whose amenability was established by Bartholdi and Vir\'ag \cite{BartholdiVirag05}. Closer to the present work is the \emph{Hanoi towers group} $\mathcal{H}^{(3)}$, introduced by Grigorchuk and \v{S}uni\'c in~\cite{Sunic, GrigorchukSunik07}. This group acts on the ternary tree and is  connected to the classical puzzle on three pegs: its Schreier graphs at level $n$ coincide with the Hanoi graphs on $n$ disks. The group $\mathcal{H}^{(3)}$ is self-similar, contracting, regular branch over its commutator subgroup, and amenable.

Several generalizations of $\mathcal{H}^{(3)}$ have been considered in the literature. Skipper \cite{Rachel} extended the construction to $d$-ary trees by taking generators whose root permutations are $(d{-}1)$-cycles, and studied the resulting branch structure, the congruence subgroup property and the Hausdorff dimension. Modified Hanoi towers groups, allowing more general patterns of non-trivial sections, were investigated by Makisumi, Stadnyk and Steinhurst \cite{modified}. In a different direction, replacing the $\mathcal{H}^{(3)}$-style relations with a Basilica-like recursion has led to the family of \emph{$p$-Basilica groups} of Di Domenico, Fern\'andez-Alcober, Noce and Thillaisundaram \cite{p-Basilica}, which provided the first examples of weakly branch (but not branch) groups with the congruence subgroup property; algorithmic aspects of the Hausdorff dimension of regular branch groups have been further studied by Fari\~na-Asategui \cite{jorge} and, jointly with Garciarena, in \cite{FG24}.

The aim of this paper is to show that, for the family introduced here, questions about an infinite group acting on a rooted tree can be systematically reduced to computations in a finite permutation group. In this way, fractality, level transitivity, the structure of the first level stabilizer, and ultimately the branch structure itself become finite, and often algorithmic, checks.

\medskip

Given the $d$-adic tree
$\T_d$, we call a group $G=\langle a_1,\dots,a_d\rangle\leq\Aut(\T_d)$ a
\emph{spinal Hanoi towers group} if each generator $a_i$ has precisely one
non-trivial section, equal to $a_i$ itself and placed at the $i$-th coordinate,
and acts at the root through a permutation $\sigma_i\in\Sym(d)$ that fixes $i$, see Figure~\ref{figura2} for the portrait of a generic generator $a_i$, and Figure~\ref{figura1} for the portraits of the generators of a concrete spinal Hanoi towers group.

\begin{figure}[H]
\begin{center}
\includegraphics[scale=0.9]{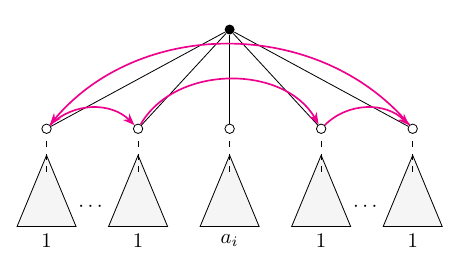}
\caption{Portrait of a generator $a_i$ with
$\psi(a_i)=(1,\dots,1,a_i,1,\dots,1)\sigma_i$, where
$\sigma_i\in\Sym(d)$ fixes $i$ and acts on the remaining
coordinates as indicated by the arrows.}
\label{figura2}
\end{center}
\end{figure}

\begin{figure}[H]
\begin{center}
\includegraphics[scale=0.8]{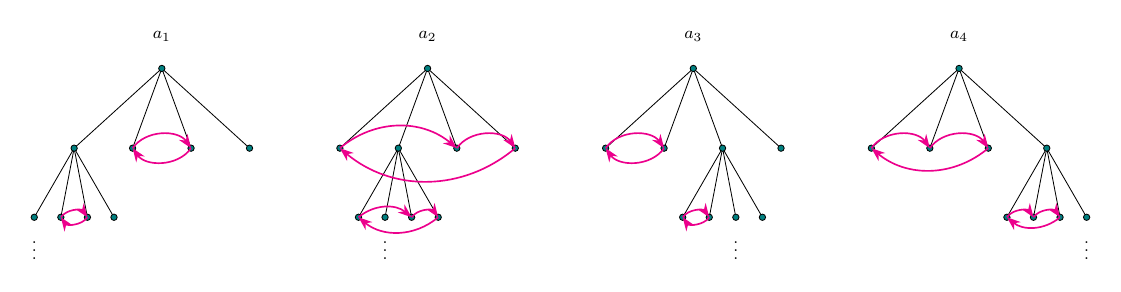}
\caption{Portraits of the generators of \(G = \langle a_1, a_2, a_3, a_4 \rangle\), where \(\psi(a_{1})=(a_{1},1,1,1)(2\ 3),\psi(a_2)=(1, a_{2},1,1)(1\ 3\ 4),\psi(a_3)=(1, 1,a_{3},1)(1\ 2),\psi(a_4)=(1, 1,1,a_{4})(1\ 2\ 3)\).}
\label{figura1}
\end{center}
\end{figure}

In this paper we present a unified framework for this family of ``Hanoi towers
group-like'' constructions. Each generator retains the rigid spinal shape
responsible for self-similarity, while the root permutations $\sigma_i$ are
arbitrary subject only to $\sigma_i(i)=i$. All the freedom of the construction is
therefore concentrated in the permutation group
$P=\langle\sigma_1,\dots,\sigma_d\rangle\leq\Sym(d)$, which makes the family
considerably larger than the one in \cite{Rachel} while still containing both the
classical Hanoi towers group $\mathcal{H}^{(3)}$ and Skipper's groups
$\mathcal{H}^{(d)}$ as extreme cases.

The generators of a spinal Hanoi towers group are directed automorphisms in the spirit of GGS-groups and of the spinal groups of Bartholdi, Grigorchuk and \v{S}uni\'c \cite{BGS03}. The distinctive feature of our family is that the directed and rooted parts are put into a single generator for each coordinate, which is precisely what makes the correspondence with the permutation group $P$ so tight.

This point of view sets up a dictionary between the theory of groups acting on
rooted trees and the theory of finite permutation groups: transitivity of $P$
characterizes fractality and level transitivity of $G$ (\Cref{fractandtrans}), the
abelianization of $P$ controls the first level stabilizer, and, when the branching subgroup $J$ introduced below is contained in $G$, perfectness of $P$
yields layered structure.

Within this family we isolate a natural subclass, the \emph{spinal Hanoi towers
groups of type $(d,m)$}, denoted $G^{(d,m)}$, in which $d\geq m+1$ and each root
permutation is the $m$-cycle $(i{+}1\ \ i{+}2\ \ \cdots\ \ i{+}m)$, with indices
taken modulo $d$. See Figure \ref{figura3} for an example of a generator of a spinal Hanoi towers group of type $(5,3)$. The two extreme cases are familiar: $m=d-1$ recovers Skipper's
groups, and $(d,m)=(3,2)$ recovers the original Hanoi towers group
$\mathcal{H}^{(3)}$.

\begin{figure}[h!]
\begin{center}
\includegraphics[scale=0.85]{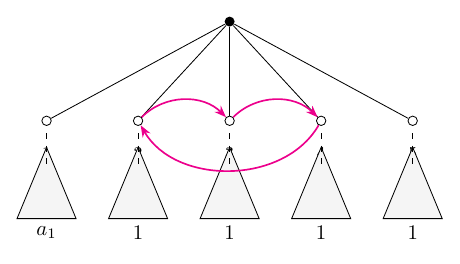}
\caption{Example of a portrait of the generator $a_1$ of a spinal Hanoi towers group of type $(5,3)$, where $\psi(a_1)=(a_1,1,1,1,1)(2 \ 3 \ 4)$.}
\label{figura3}
\end{center}
\end{figure}

The first part of the paper develops the theory for an arbitrary spinal Hanoi towers group $G$.  We start by showing that $G$ is fractal and level transitive
precisely when $P$ acts transitively on $\{1,\dots,d\}$
(\Cref{fractandtrans}). The recursive
structure also yields good algorithmic control: every spinal Hanoi towers group is contracting (\Cref{Proposition: Contraction}). The key observation is that, if one measures an element by its syllable length with respect to the powers of the generators, then passing to first level sections strictly shortens it. As a consequence, we obtain a  recursive procedure that solves the word problem (\Cref{alg:word-problem}). We stress that decidability of the word problem follows from the general theory of contracting automaton groups. On the other hand,  the content of our approach is the explicit length-halving bound on the syllable length, which is also what yields the contracting property with an explicit nucleus. Finally, we identify the abelian invariants of $G$.
The total exponent with which each generator occurs gives a homomorphism
$\varepsilon\colon G\to\prod_{i=1}^d\mathbb{Z}/o(a_i)\mathbb{Z}$ inducing an
isomorphism $G/G'\cong\prod_{i=1}^d\mathbb{Z}/o(a_i)\mathbb{Z}$
(\Cref{epsilonmap}), and this exponent map is compatible with the self-similar
structure, in the sense that $\varepsilon(g)=\sum_{j=1}^d\varepsilon(g|_j)$ for
every $g\in\St_G(1)$ (\Cref{cor:epsilon-sections}). This compatibility is what
makes the exponent map a useful tool that will be used throughout the rest of the paper. As a first application we combine it with the action on the first level to describe the
first level stabilizer: a tuple $(g_1,\dots,g_d)\in G\times\cdots\times G$ lies in
$\psi(\St_G(1))$ if  its total exponent lies in the kernel of an
explicit homomorphism
$\lambda\colon\prod_{i=1}^d\mathbb{Z}/o(a_i)\mathbb{Z}\to P/P'$ (\Cref{prop:necessary-condition-first-stabilizer} and
\Cref{prop:sufficient-condition-first-stabilizer}).\\

For the reader’s convenience, we record these results together in the following theorem.
\begin{thmx}\label{thmA}
Let $G=\langle a_1,\dots,a_d\rangle$ be a spinal Hanoi towers group with associated permutation group $P=\langle\sigma_1,\dots,\sigma_d\rangle\leq\Sym(d)$. Then:
\begin{enumerate}
\item $G$ is fractal and level transitive if and only if $P$ is transitive on $\{1,\dots,d\}$;
\item $G$ is amenable and contracting, with nucleus $\{1\}\cup\{a_i^r\mid 1\leq i\leq d,\ 1\leq r<o(a_i)\}$, and the word problem in $G$ is solvable by an explicit recursion on the syllable length;
\item $G/G'\cong\prod_{i=1}^d\mathbb{Z}/o(a_i)\mathbb{Z}$;
\item $\psi(\St_G(1))
\subseteq
\left\{
(g_1,\ldots,g_d)\in G\times \stackrel{d}{\cdots}\times G
\ \middle|\
\sum_{j=1}^d \varepsilon(g_j)\in \ker(\lambda)
\right\}.$
\end{enumerate}
\end{thmx}

\medskip
Throughout the paper, we denote by $J\leq \Aut(\T_d)$ the subgroup
generated by the elements of $\St(1)$ whose wreath recursion has exactly two
non-trivial coordinates: for some distinct $p,q\in\{1,\ldots,d\}$, these
coordinates are occupied by an element $h\in G$ and its inverse $h^{-1}$,
respectively.

The subgroup $J$ governs the branch structure of a spinal Hanoi towers
group. Although its generators are defined using elements of $G$, they are
a priori automorphisms of $\T_d$ and need not themselves belong to $G$.
The central question is therefore to determine when $J\leq G$. For a
level-transitive group $G$, \Cref{Transfer} gives a criterion for this
containment in terms of the image of $\St_G(1)$ in the finite group
\[
\left(\prod_{i=1}^d \mathbb{Z}/o(a_i)\mathbb{Z}\right)^d\rtimes P,
\]
and \Cref{alg:Jcriterion} turns this criterion into an explicit algorithm.
When $J\leq G$, the consequences are strong: $G$ is strongly fractal and
regular branch over $G'$
(\Cref{lemma:J-contains-derived-sections}); its first-level stabilizer is
described exactly by the exponent-sum obstruction
(\Cref{prop:sufficient-condition-first-stabilizer}); and its rigid
stabilizers admit an explicit description at every level
(\Cref{prop:rist-explicit}). This description yields finite abelianizations
of the rigid stabilizers, and Grigorchuk's criterion for just infiniteness
then implies that $G$ is just infinite
\cite{JI grigorchuk} (\Cref{thm:justinfinite}).

Again, for the reader’s convenience, we record the structural consequences of the condition $J\leq G$ together in the following theorem.

\begin{thmx}\label{thmB}
Let $G$ be a level transitive spinal Hanoi towers group and let $J$ be the subgroup defined above. If $J\leq G$, then:
\begin{enumerate}
    \item $G$ is strongly fractal;
    \item $G$ is  regular branch over $G'$;
    \item $\psi(\St_G(1))
=
\left\{
(g_1,\ldots,g_d)\in G\times \stackrel{d}{\cdots}\times G \ \middle| \sum_{j=1}^d \varepsilon(g_j)\in \ker(\lambda)
\right\};$
 \item $\psi_m(\Rist_G(m))=B\times\stackrel{d^m}{\cdots}\times B$ for every $m\geq 1$, where $B=\rho^{-1}(P')$ (see \Cref{therhomap} for the definition of $\rho$);
    \item $G$ is just infinite.
\end{enumerate}
Note that, for  a spinal Hanoi towers group $G$ that is regular branch over $G'$, the containment $J\leq G$ can be verified in the finite group $\bigl(\prod_{i=1}^d\mathbb{Z}/o(a_i)\mathbb{Z}\bigr)^d\rtimes P$.
\end{thmx}

\medskip
In the final part of the paper, we specialize to the spinal Hanoi towers
groups $G^{(d,m)}$ of type $(d,m)$. For every $d\geq 3$ and
$2\leq m\leq d-1$, we prove that $G^{(d,m)}$ is strongly fractal and
regular branch over its commutator subgroup
(\Cref{propstronglyfract}). We then determine exactly when the subgroup
$J$ is contained in $G^{(d,m)}$: this occurs if and only if
$d\geq 4$ (\Cref{IcontainedinG'}). In the same range, the general
criterion developed above implies that $G^{(d,m)}$ is just infinite;
indeed, $G^{(d,m)}$ is just infinite if and only if $d\geq 4$
(\Cref{thm:typem-just-infinite}).

The exceptional case $d=3$ necessarily corresponds to
$(d,m)=(3,2)$, and hence to the classical Hanoi towers group
$\mathcal{H}^{(3)}$. In this case $J$ is not contained in the group
and the group is not just infinite. Thus, for every $d\geq 4$, our
methods give a uniform description of the branch structure of all groups
of type $(d,m)$.\\

For the reader's convenience, we collect these conclusions in the
following theorem.
\begin{thmx}\label{thmC}
Let $d\geq 3$ and $2\leq m\leq d-1$, and let $G^{(d,m)}$ be the spinal Hanoi towers group of type $(d,m)$. Then $G^{(d,m)}$ is strongly fractal and regular branch over its commutator subgroup.
Moreover, the additional properties of $J\leq G^{(d,m)}$  and $G^{(d,m)}$ being just infinite hold if and only if $d\geq 4$.

\end{thmx}
Our results suggest that the containment $J\leq G$ may in fact characterize just infiniteness within the family (see Question~\ref{q:general}). A positive answer would identify the Hanoihedral group of Garrido and \v{S}uni\'c \cite{GarridoSunic2025} as a new natural example of a Hanoi-type group whose branch structure falls short of just infiniteness.

\medskip

\medskip

\noindent\textbf{Organization of the paper.}
In \Cref{sec:preliminaries} we collect the prerequisite notions on groups
acting on rooted trees and recall some classical examples. In
\Cref{sec:definition} we introduce spinal Hanoi towers groups, discuss their
relation with the constructions of Skipper \cite{Rachel} and
Makisumi-Stadnyk-Steinhurst \cite{modified}, and establish their basic
properties, stated in \Cref{thmA}. In
\Cref{sectbranch} we study the first level stabilizer of a spinal Hanoi towers
group and obtain necessary and sufficient conditions for a tuple to belong to it. Section \ref{secjustinf} is devoted to proving that level transitive spinal Hanoi towers groups with $J\le G$ are just infinite; in particular, we conclude the proof of \Cref{thmB}. Finally, in \Cref{sec:examples}, we specialize to the groups of type~$(d,m)$, and we complete the proof of \Cref{thmC}.

\medskip 

\subsection*{Acknowledgements} The authors thank J. Fari\~{n}a-Asategui and G. A. Fernández-Alcober for useful discussions.

\section{Preliminaries}\label{sec:preliminaries}

Let $X$ be a finite alphabet of cardinality $d \geq 2$. The \emph{$d$-regular rooted tree} $\T_d$ is the tree whose vertex set is the set $X^*$ of all finite words over $X$ (including the empty word $\emptyset$, which is the root), and where a vertex $v$ is connected by an edge to $vx$ for each $x \in X$. The \emph{$n$-th level} of the tree is the set $X^n$ of words of length $n$. The \emph{boundary} of the tree, denoted $\partial \T_d$, is the set $X^\omega$ of right-infinite sequences over $X$, endowed with the product topology.

The group $\Aut(\T_d)$ of automorphisms of $\T_d$ consists of all  automorphisms that fix the root. Every automorphism preserves the level structure of the tree. The group $\Aut(\T_d)$ admits a natural recursive description via the \emph{wreath product decomposition}
\[
  \Aut(\T_d) \cong \Aut(\T_d) \wr \Sym(X) = \Aut(\T_d)^d \rtimes \Sym(X).
\]

The wreath product decomposition above is realized by the isomorphism
\[
  \psi\colon \Aut(\T_d) \longrightarrow  \Aut(\T_d)\wr \Sym(X)=\Aut(\T_d)^d\rtimes \Sym(X),
\]
which sends an automorphism $g\in \Aut(\T_d)$ to the tuple of its first level sections together with its root permutation, that is
\[
  \psi(g)=(g|_{x_1},g|_{x_2},\ldots,g|_{x_d}) \sigma,
\]
where $X=\{x_1,\ldots,x_d\}$ and $\sigma\in \Sym(X)$ is the root permutation of $g$. For each $x\in X$, the projection
\[
  \psi_x\colon \st(x) \longrightarrow  \Aut(\T_d),\qquad g \longmapsto  g|_x,
\]
restricted to the stabilizer of the vertex $x$, is a group homomorphism. For a self-similar subgroup $G\leq \Aut(\T_d)$, the wreath recursion closes within $G$ and yields the embedding
\[
  \psi\colon G \hookrightarrow  G\wr \Sym(X),
\]
which we shall use throughout the paper without further mention.

In other words, every automorphism $g \in \Aut(\T_d)$ can be written as
\[
  \psi(g) = (g_1, g_2, \ldots, g_d)\sigma,
\]
where $\sigma \in \Sym(X)$ is the \emph{root permutation} (or \emph{activity}) of $g$, describing how $g$ permutes the top-level subtrees, and $g_x \in \Aut(\T_d)$ is the \emph{section} (or \emph{state}) of $g$ at the vertex $x \in X$, describing the action of $g$ on the subtree rooted at $x$. The action on a word $xw$, with $x \in X$ and $w \in X^*$, is given by
\[
  g(xw) = \sigma(x)  g_x(w).
\]
More generally, the section of $g$ at a vertex $v = x_1 x_2 \cdots x_n$ is defined recursively by $g|_v = (g|_{x_1 \cdots x_{n-1}})|_{x_n}$, with $g|_\emptyset = g$.
Throughout the paper, automorphisms act on the right and composition is read from  left to right. We write $u^g$ for the image of a vertex $u$ under $g$, so that $u^{gh}=(u^g)^h$, and accordingly $(gh)|_u=g|_u\, h|_{u^g}$. All displayed computations below follow this convention.
We will use the following formula for the section of a conjugate without further reference. For $f,g\in  \Aut(\T_d)$ and a vertex $u$, we have
\begin{equation}
    (f^{g})|_{u}=(g|_{(u)^{g^{-1}}})^{-1}f|_{(u)^{g^{-1}}}g|_{(u)^{g^{-1}f}}.
\end{equation}

\begin{definition}
  The \emph{stabilizer} of the $n$-th level is the normal subgroup
  \[
    \St(n) = \{ g \in \Aut(\T_d) : g(v) = v \text{ for all } v \in X^n \}.
  \]
  For a subgroup $G \leq \Aut(\T_d)$, we write $\St_G(n) = G \cap \St(n)$.
\end{definition}

\begin{definition}
 The \emph{rigid stabilizer} of a vertex $v$ is the subgroup
   \[
     \rist(v) = \{ g \in \Aut(\T_d) : g(w) = w \text{ for all } w \notin v X^* \},
  \]
   consisting of all automorphisms that act trivially outside the subtree rooted at $v$. The \emph{rigid stabilizer of the $n$-th level} is
   \[
     \Rist(n) = \prod_{v \in X^n} \rist(v).
   \]
   In $\Aut(\T_d)$ we have $\St(n)=\Rist(n)$ for every $n$. On the other hand, for a subgroup $G \leq \Aut(\T_d)$, we set $\rist_G(v) = G \cap \rist(v)$ and $\Rist_G(n) = \prod_{v \in X^n} \rist_G(v)$.
 \end{definition}

\begin{definition}
  A subgroup $G \leq \Aut(\T_d)$ is \emph{self-similar} if for every $g \in G$ and every $x \in X$, the section $g|_x$ also belongs to $G$. %Equivalently, $G$ is self-similar if the wreath recursion closes within $G$.
 \end{definition}

\begin{definition}\label{defstabvertex}
Let $G\leq \Aut(\T_d)$ be self-similar. We say that $G$ is \emph{fractal} if $\psi_{x}(\st_{G}(x))=G$ for every vertex $x$, where $\st_G(x)$ denotes the stabilizer of $x$ in $G$, and that $G$ is \emph{strongly fractal} if $\psi_{x}(\St_{G}(1))=G$ for every $x\in X$. 
\end{definition}

\begin{definition}
  A self-similar group $G \leq \Aut(\T_d)$ is \emph{contracting} if there exists a finite set $\mathcal{N} \subset G$ such that for every $g \in G$ there exists $n_0$ with the property that $g|_v \in \mathcal{N}$ for all $v$ with $|v| \geq n_0$. The smallest such set $\mathcal{N}$ is called the \emph{nucleus} of $G$.
\end{definition}

\begin{definition}
   Let $G \leq \Aut(\T_d)$. We say that $G$ is \emph{level transitive} if it acts transitively on each level.
\end{definition}

\begin{definition}
A subgroup $G \leq \Aut(\T_d)$ acting level transitively is called:
  \begin{itemize}
    \item \emph{branch} if $\Rist_G(n)$ has finite index in $G$ for every $n \geq 0$;
    \item \emph{weakly branch} if $\Rist_G(n)$ is non-trivial for every $n \geq 0$.
  \end{itemize}
\end{definition}

\begin{definition}
  A group $G \leq \Aut(\T_d)$ is \emph{regular branch} over a subgroup $K \leq G$ if $G$ acts level transitively, $K$ has finite index in $G$, and $K \times \cdots \times K$ ($d$ copies) embeds into $K$ via the wreath recursion.
\end{definition}

\section{Definition of the group and first properties}\label{sec:definition}

In this section we define the class of spinal Hanoi towers groups and establish some of its main properties. We recall that the \emph{Hanoi towers group} $\mathcal{H}^{(3)}$, introduced by Grigorchuk and \v{S}uni\'{c}~\cite{GrigorchukSunik07}, acts on the ternary tree $\T_3$  and is generated by three automorphisms $a_1,a_2,a_3$ where 
\begin{align*}
\psi(a_1)&=(a_1,1,1)(2\ 3),\\
\psi(a_2)&=(1,a_2,1)(1\ 3),\\
\psi(a_3)&=(1,1,a_3)(1\ 2).
\end{align*}
  The name derives from its connection to the classical \emph{Tower of Hanoi puzzle} on three pegs: the Schreier graph of the action of $\mathcal{H}^{(3)}$ on $X^n$ is isomorphic to the state graph of the Hanoi puzzle with $n$ disks. The group $\mathcal{H}^{(3)}$ is self-similar, contracting, and regular branch over its commutator subgroup. 

A first generalization to $d$-adic trees, for $d\geq 3$, is due to Skipper~\cite{Rachel}, who introduced the group $\mathcal{H}^{(d)}=\langle a_1,\dots,a_d\rangle\leq \Aut(\T_d)$, where each $a_i$ has a single non-trivial section, equal to $a_i$ itself, placed at the $i$-th coordinate, and root permutation
\[
\sigma_i=(i+1\   i+2\ \, \cdots\ \, i+d-1)\in \Sym(d),
\]
i.e.\ the $(d-1)$-cycle on $\{1,\dots,d\}\setminus\{i\}$ (indices mod $d$). She studied the branch structure, the congruence subgroup property and the Hausdorff dimension of these groups; for $d=3$ one recovers the classical Hanoi towers group $\mathcal{H}^{(3)}$. In~\cite{modified}, Makisumi, Stadnyk and Steinhurst introduced a different generalization in which a generator $a$ may have more than one non-trivial section (each equal to $a$); see~\cite[Definition~4.1]{modified}. They studied the contracting property and obtained necessary and sufficient conditions for such groups to be contracting~\cite[Proposition~4.4, Theorem~4.5]{modified}. Modifying instead the wreath recursion in a Basilica-like fashion leads to the family of $p$-Basilica groups of~\cite{p-Basilica}.

In this paper, we let $\T_d$ be the $d$-ary tree. For $i\in \{1,\dots, d\}$, consider the elements
\begin{align*}
\psi(a_i)=(1,\stackrel{i-1}{\dots},1,a_i,1,\dots, 1)\sigma_i,
\end{align*}
of the group $\Aut (\T_d)$,  where $\sigma_i$ is an element of the symmetric group $\Sym(d)$ that fixes $i$. We refer to any group generated by elements $a_1,\dots, a_d$ as a spinal Hanoi towers group. 
We assume throughout that $\sigma_{i}\neq 1$ for every $i$, in order to exclude the degenerate case $a_{i}=1$.

Of course, the group $\mathcal{H}^{(3)}$ and, more generally, the groups $\mathcal{H}^{(d)}$ of Skipper~\cite{Rachel} are examples of spinal Hanoi towers groups.

We now characterize when a spinal Hanoi towers group is fractal and level transitive.

\begin{proposition}\label{fractandtrans} Let $G$ be a spinal Hanoi towers group acting on the $d$-ary tree $\T_d$. Then $G$ is fractal and level transitive if and only if $\langle\sigma_1,\dots,\sigma_d\rangle$ acts transitively on the set $X=\{1,2,\dots,d\}$.
\end{proposition}
\begin{proof}
If $P:=\langle \sigma_1,\ldots,\sigma_d\rangle$ is not transitive on
$X=\{1,\ldots,d\}$, then $G$ is not transitive on the first level, hence
it is not level transitive.

Conversely, assume that $P$ is transitive, that is $G$ is transitive on the first level. 
By \cite[Lemma 2.7]{fractality}, it is
enough to prove that $\psi_1(\st_G(1))=G$, where $\st_{G}({1})$ is the stabilizer of the vertex 1 and not the stabilizer of the first level. Since
$\sigma_1(1)=1$ and $a_1|_1=a_1$, we have
$a_1\in \psi_1(\st_G(1))$.

Let now $i\neq 1$. Choose a word
\[
        \delta=\sigma_{i_1}\sigma_{i_2}\cdots \sigma_{i_m}
\]
of minimal length such that $\delta(i)=1$, and set
$b=a_{i_1}a_{i_2}\cdots a_{i_m}$. We claim that $a_i^b\in
\St_G(1)$ and $(a_i^b)|_1=a_i$.

Indeed, the root permutation of $a_i^b$ is $\sigma_i^\delta$, and hence
\[
        \sigma_i^\delta(1)
        =\delta^{-1}\sigma_i\delta(1)
        =\sigma_i\delta(i)
        =\delta(i)
        =1,
\]
because $\sigma_i$ fixes $i$. Thus $a_i^b$ fixes the vertex $1$, i.e. $a_i^b\in
\St_G(1)$. It remains to compute its section at $1$. Since $\delta^{-1}(1)=i$, we have
\[
        (a_i^b)|_1
        = (b|_i)^{-1} a_i|_i\, b|_i
        = (b|_i)^{-1} a_i b|_i .
\]
We show that $b|_i=1$. For $j=1,\ldots,m$, put
\[
        x_j=(\sigma_{i_1}\cdots\sigma_{i_{j-1}})(i),
\]
with $x_1=i$. If $x_j=i_j$ for some $j$, then $\sigma_{i_j}$ fixes
$x_j$, and deleting the letter $\sigma_{i_j}$ from the word $\delta$
would still give a word sending $i$ to $1$, contradicting the minimality
of $\delta$. Hence $x_j\neq i_j$ for every $j$. But $a_{i_j}$ has
non-trivial section only at the vertex $i_j$, so
\[
        a_{i_j}|_{x_j}=1
\]
for every $j$. Therefore $b|_i=1$, and consequently
\[
        (a_i^b)|_1=a_i.
\]
Thus $a_i\in \psi_1(\St_G(1))$, which proves that the group is fractal and level transitive.
\end{proof}

\begin{remark}
	\Cref{fractandtrans} allows us to reduce many questions concerning a spinal Hanoi towers group $G$ to the case in which
	the group generated by the root permutations acts transitively. More
	precisely, suppose that $G$ acts on the $d$-adic tree $\T_d$ and that there is a partition of $X=\{1,\ldots,d\}$ given by,
	\[
	X=X_1\sqcup\cdots\sqcup X_r
	\]
	such that, whenever $i\in X_j$, the permutation $\sigma_i$ fixes
	$X\setminus X_j$ pointwise. For each $j$, set
	\[
	G_j=\langle a_i\mid i\in X_j\rangle
	\qquad\text{and}\qquad
	P_j=\langle \sigma_i|_{X_j}\mid i\in X_j\rangle
	\leq \Sym(X_j).
	\]
	Then $G_j$ fixes pointwise every first-level subtree whose root does
	not belong to $X_j$. Consequently, the subgroups $G_1,\ldots,G_r$
	commute, their supports are disjoint, and the multiplication map induces
	an isomorphism
	\[
	G\cong G_1\times\cdots\times G_r.
	\]
	Moreover, restriction to the invariant subtree $X_j^*$ identifies
	$G_j$ with the spinal Hanoi towers group acting on the
	$\lvert X_j\rvert$-ary tree whose root permutations are
	$\sigma_i|_{X_j}$, for $i\in X_j$. In particular, if $P_j$ acts
	transitively on $X_j$, then Proposition~3.1 shows that $G_j$ is
	fractal and level transitive.
	
	Thus, whenever the wreath recursions split into disjoint transitive
	blocks in this way, the resulting non-level transitive group introduces
	no new indecomposable example: it is simply a direct product of
	level transitive spinal Hanoi towers groups.
\end{remark}
The following example illustrates this situation with two blocks of cardinality $3$.
\begin{example}
	Let $G=\langle a_1,\ldots,a_6\rangle\leq\Aut(\T_6)$, where
	\[
	\begin{aligned}
		a_1&=(a_1,1,1,1,1,1)(23),&
		a_2&=(1,a_2,1,1,1,1)(13),&
		a_3&=(1,1,a_3,1,1,1)(12),\\
		a_4&=(1,1,1,a_4,1,1)(56),&
		a_5&=(1,1,1,1,a_5,1)(46),&
		a_6&=(1,1,1,1,1,a_6)(45).
	\end{aligned}
	\]
	Taking
	\[
	X_1=\{1,2,3\}
	\qquad\text{and}\qquad
	X_2=\{4,5,6\},
	\]
	the  remark above gives
	\[
	G
	=\langle a_1,a_2,a_3\rangle
	\times
	\langle a_4,a_5,a_6\rangle
	\cong \mathcal{H}^{(3)}\times \mathcal{H}^{(3)}.
	\]
	Indeed, after relabelling the elements of $X_2$, both factors are
	isomorphic to the classical Hanoi towers group $\mathcal{H}^{(3)}$.
\end{example}

\begin{remark}\label{remarkamenabke}
	Every generator $a_i=(1,\dots,a_i,\dots,1)\sigma_i$ has a single non-trivial
	section fixed by $\sigma_i$. Consequently, each
	$a_i$ is a bounded automorphism of $\T_d$, and every spinal Hanoi towers
	group is generated by bounded automata. By the amenability of bounded automata groups~\cite{BartholdiKaimanovichNekrashevych10}, every spinal Hanoi towers group
	is amenable. In particular this recovers the amenability of $\mathcal{H}^{(3)}$
	and of Skipper's groups $\mathcal{H}^{(d)}$.
\end{remark}

\subsection{Contracting property and the word problem}\label{sec:contracting}

In this section we prove that every spinal Hanoi towers group is contracting and provide a recursive algorithm solving the word problem. To this end, we note that apart from the classical generating set $\{a_1,\dots,a_d\}$ of a spinal Hanoi towers group $G$, another (non-minimal) generating set of $G$ is given by 
\[
\mathcal{S}:=\{ a_i^r \mid 1\le i\le d,\ 1\le r< o(a_i) \}\subseteq G, 
\]
where $o(x)$ denotes the order of an element $x$ of a group $G$, that is the smallest positive integer $k$ such that $x^k=1$. It is clear that $o(a_i)=o(\sigma_i)<\infty$, since $a_i^r=(1,\dots,1,a_i^r,1,\dots,1)\sigma_i^r$ by induction, and an automorphism with trivial root permutation whose unique non-trivial section equals the element itself is trivial.
We say a word $w$ over the alphabet $\mathcal{S}$ is of \emph{syllable form} if it has the form
\[
w  =  a_{i_1}^{ r_1} a_{i_2}^{ r_2}\cdots a_{i_k}^{ r_k},
\]
for some $k\in \mathbb{N}\cup\{0\}$ and some $a_{i_j}^{ r_j}\in \mathcal{S}$ with $i_j\neq i_{j+1}$ for $j\in\{1,\ldots,k-1\}$. The minimal integer $k$ among all such products evaluating  to a given element $g\in G$ is called the \emph{syllable length} of $g$ and is denoted $\Syl(g)$, and any corresponding expression is called a \emph{minimal syllable form of $g$}. Thus the syllable length is the (weighted) word length with respect to the generating set $\mathcal{S}$ with elements of $\mathcal{S}$, that we shall call \emph{syllables}, having weight 1. 
%Note that the corresponding distance function is a pseudo-metric on $G$. 

This remark will be useful for the next proposition.

\begin{remark}
Note that since the syllables are powers of generators, looking at the sections we have the same power of the same generator in exactly one section and 1 otherwise. 
\end{remark}

\begin{proposition}\label{Proposition: Contraction} Let $G$ be a spinal Hanoi towers group acting on the $d$-ary tree, and let $I=\{1, \dots, d\}$. 
\begin{enumerate}
    \item Let $g\in G$. Then, for every $i\in I$, $\Syl(g|_i) \leq \lceil \Syl(g)/2 \rceil$.
    \item Let $g\in G$. Then $\sum_{i\in I} \Syl(g|_i) \leq \Syl(g)$.
\end{enumerate}
\end{proposition}
\begin{proof}Let $g$ be an element of $G$. We first treat the cases of syllable length 1 and 2, then reduce the general case $\Syl(g) =k\geq 3$ by splitting the word into blocks of length at most 2. If $\Syl(g)=1$, then it follows that $g=a_i^r$ for some $i\in I$ and some $r\in \{1,\ldots, o(a_i)-1\}$. Then $\psi(g)=(1,\ldots,1,a_i^r,1,\ldots,1)\sigma_i^r$, and (i) is true. If $\Syl(g)=2$, then $g=a_{i_1}^{r_1}a_{i_2}^{r_2}$ for some $i_j\in I$ with $i_1\not=i_2$ and some $r_j\in \{1,\ldots, o(a_{i_j})-1\}$. Therefore 
\[
\psi(g)=(1,\ldots,1, a_{i_1}^{r_1},1,\ldots,1,a_{i_2}^{r_2},1,\ldots,1)\sigma_{i_1}^{r_1} \sigma_{i_2}^{r_2}
\]
where $a_{i_1}^{r_1}$ lies in the $i_1$-th coordinate and $a_{i_2}^{r_2}$ lies in the $i_2^{\sigma_{i_1}^{-r_1}}$-th coordinate. Note that $i_2^{\sigma_{i_1}^{-r_1}}\not=i_1$ since $\sigma_{i_1}$ fixes $i_1$ and $i_1\not=i_2$. So (i) is true. Now assume that $\Syl(g)=k\geq 3$. Then
\[
g=a_{i_1}^{ r_1} a_{i_2}^{ r_2}\cdots a_{i_k}^{ r_k}
\]
for $i_j\in I$ with $i_{j}\not=i_{j+1}$ for $j\in \{1,\ldots,k-1\}$ and some $r_j\in \{1,\ldots,o(a_{i_j})-1\}$. Note that $g$ can be written as a product $h_1\cdots h_{\lceil k/2 \rceil}$ where each $h_i$ has syllable length at most 2. Since
\[
\psi(h_i)=(h_i|_1,\ldots, h_i|_d)\tau_i
\]
for some $\tau_i\in \Sym(d)$ and, by the cases  treated above, $\Syl(h_i|_j)\leq 1$, it follows that for every $i\in \{1,\ldots,d\}$ the section $g|_i$ is a product of $\lceil k/2 \rceil$ syllables of length at most 1, therefore $\Syl(g|_i)\leq \lceil k/2 \rceil$.

Lastly, (ii) is a consequence of the fact that every syllable has exactly one non-trivial section and therefore contributes at most one syllable to exactly one of the sections, allowing only cancellations and no creation of new syllables. This concludes the proof.
\end{proof}

We now introduce some notation that will be used throughout the paper to keep track of
the exponents of the generators. This will be useful whenever we compare an element with
its sections, pass to quotients of $G$, or study the structure of the first level stabilizer.
For the moment, the notation is attached to a chosen syllable form expression; later we will
show that it is independent of this choice and hence defines a homomorphism on $G$.

Let $G=\langle a_1,\ldots,a_d\rangle$ be a spinal Hanoi towers group.
Given a syllable form expression
\[
\gamma  =  a_{i_1}^{r_1} a_{i_2}^{r_2}\cdots a_{i_k}^{r_k},
\]
we define, for each $i\in \{1,\ldots,d\}$,
\[
e_i(\gamma)
:=
\sum_{\substack{1\le t\le k\\ i_t=i}} r_t
\pmod {o(a_i)}
\in \mathbb Z/o(a_i)\mathbb Z.
\]
Thus $e_i(\gamma)$ is the total exponent with which the generator $a_i$ occurs in
the expression $\gamma$, taken modulo the order of $a_i$. We also define the exponent
vector of $\gamma$ by
\[
\varepsilon(\gamma)
:=
\bigl(e_1(\gamma),\ldots,e_d(\gamma)\bigr)
\in
\prod_{i=1}^d \mathbb Z/o(a_i)\mathbb Z.
\]
This notation is compatible with taking first level sections. Indeed, if
\[
\psi(\gamma)=(\gamma|_1,\ldots,\gamma|_d)\sigma,
\]
then, for every $i\in \{1,\ldots,d\}$,
\begin{equation}\label{Observation: sum of exponets over sections}
    e_i(\gamma)=\sum_{j=1}^d e_i(\gamma|_j).
\end{equation}
To see this, observe that each syllable $a_\ell^r$ has exactly one non-trivial first level section, namely $a_\ell^r$, and all its other sections are trivial. Therefore each occurrence of $a_i^r$ in $\gamma$ contributes $r$ to exactly one of the sections $\gamma|_j$, and contributes nothing to the others. Summing over all syllables gives the desired identity.

The previous identity shows that the exponent information of a word can be recovered by looking at the exponent information of its first level sections. This is useful because sections are shorter than the original word in the sense of syllable length, by
\Cref{Proposition: Contraction}. Thus one can test whether a word represents the identity by first checking its action on the first level and then recursively checking its sections. This gives the following algorithm for the word problem. 

\begin{algorithm}[H]
\caption{Word problem for $G$ via sections}\label{alg:word-problem}
\begin{algorithmic}[1]
\Require $G = \langle a_1, \dots, a_d \rangle$ a spinal Hanoi towers group; $\gamma = a_{i_1}^{ r_1}\cdots a_{i_k}^{ r_k}$ in syllable form ($i_j \neq i_{j+1}$)
\Ensure \textsc{True} if $\gamma = 1$ in $G$, \textsc{False} otherwise
\Procedure{IsIdentity}{$\gamma$}
  \State \textbf{case} $\Syl(\gamma) = 0$: \Return \textsc{True}
  \State \textbf{case} $\Syl(\gamma) = 1$, $\gamma = a_i^r$: \Return $r \equiv 0 \pmod{o(a_i)}$
  \State Compute $\psi(\gamma) = (\gamma|_1, \dots, \gamma|_d) \sigma \in (G\times\stackrel{d}{\cdots}\times G) \rtimes \Sym(d)$
  \State \textbf{if} $\sigma \neq 1$ \textbf{then} \Return \textsc{False}
  \State \Return $\bigwedge_{i=1}^{d}$ \Call{IsIdentity}{$\gamma|_i$}
\EndProcedure
\end{algorithmic}
\medskip
\end{algorithm}

Note that by \Cref{Proposition: Contraction}, $\Syl(\gamma|_i) \leq \lceil \Syl(\gamma)/2 \rceil < \Syl(\gamma)$ whenever $\Syl(\gamma) \geq 2$, so the algorithm terminates.

\begin{corollary}\label{cornucleo}Spinal Hanoi towers groups are contracting with nucleus
\[
\mathcal{N}\coloneqq\{1, a_i^r \mid 1\le i\le d,\ 1\le r\le o(a_i)-1\}.
\]
\end{corollary}
\begin{proof}
By \Cref{Proposition: Contraction}(i), for every $g\in G$ there exists $n_0$ such that for all $|v|\geq n_0$, we have $\Syl(g|_v)\leq 1$, that is, $g|_v\in\mathcal{N}$. Hence the nucleus is contained in $\mathcal{N}$. Conversely, we note that  the section of $a_i^r$ at the vertex $i\stackrel{n}{\cdots}i$ equals $a_i^r$ for every $n\geq 0$. Therefore $\mathcal{N}$ is contained in the nucleus, and the equality follows.
\end{proof}

\begin{proposition}\label{epsilonmap}
	Let $G$ be a spinal Hanoi towers group. Then the map
	\begin{align*}
		\varepsilon: G&\longrightarrow \prod_{j=1}^d \mathbb Z/o(a_j)\mathbb Z,\\
		\gamma=a_{i_1}^{ r_1} a_{i_2}^{ r_2}\cdots a_{i_k}^{ r_k}&\longmapsto\bigl(e_1(\gamma), \ldots, e_d(\gamma)\bigr)
	\end{align*}
	is a well-defined surjective group homomorphism. Furthermore, $\ker(\varepsilon)=G'$ and 
	\[
	G/G'\cong \prod_{j=1}^d \mathbb Z/o(a_j)\mathbb Z.
	\]
\end{proposition}
\begin{proof}
	Let $\gamma_1$ and $\gamma_2$ be two syllable forms that evaluate to the same element $g\in G$,
	and set $\gamma:=\gamma_1\gamma_2^{-1}$, so that $\gamma=1$ in $G$.
	We show, by induction on $\Syl(\gamma)$, that if a syllable form $\gamma$ evaluates to $1$ in $G$, then $e_i(\gamma)\equiv 0\pmod{o(a_i)}$ for every $i$. If $\Syl(\gamma)\leq 1$ this is clear. If $\Syl(\gamma)\geq 2$, by  $\gamma=1$,  we have $\gamma|_j=1$, and the syllable form of each section satisfies $\Syl(\gamma|_j)\leq\lceil\Syl(\gamma)/2\rceil<\Syl(\gamma)$ by \Cref{Proposition: Contraction}.
    By the induction hypothesis, $e_i(\gamma|_j)\equiv 0$ for all $i,j$, and Formula~\eqref{Observation: sum of exponets over sections} gives $e_i(\gamma)=\sum_{j} e_i(\gamma|_j)\equiv 0 \pmod{o(a_i)}$.
	Therefore $\varepsilon(\gamma_1)=\varepsilon(\gamma_2)$, and hence $\varepsilon$ is well defined.
	It is immediate from the definition that $\varepsilon$ is a homomorphism, and it is surjective
	because the images $\varepsilon(a_1),\dots,\varepsilon(a_d)$ generate
	$\prod_{j=1}^d \mathbb Z/o(a_j)\mathbb Z$.
    To conclude we show that $\ker(\varepsilon)=G'$. Since the codomain is abelian, the inclusion $G'\leq \ker(\varepsilon)$ is obvious.
    Conversely, let $\gamma\in \ker(\varepsilon)$ and write
    \[
\gamma=a_{i_1}^{r_1}a_{i_2}^{r_2}\cdots a_{i_k}^{r_k}.
   \]
   In the abelian quotient $G/G'$ we may reorder the factors and collect exponents, obtaining
   \[
   \gamma G'=\prod_{j=1}^d a_j^{e_j(\gamma)}G'.
   \]
   Since $\gamma\in \ker(\varepsilon)$, we have $e_j(\gamma)\equiv 0 \pmod{o(a_j)}$ for every
   $j\in\{1,\ldots,d\}$, and hence $a_j^{e_j(\gamma)}=1$ in $G$ for every $j$. Therefore
   $\gamma G'=G'$, i.e.\ $\gamma\in G'$. This proves $\ker(\varepsilon)=G'$, and,  by the first isomorphism theorem,
   \[
G/G'=G/\ker(\varepsilon)\cong\prod_{j=1}^d \mathbb Z/o(a_j)\mathbb Z.
   \]
\end{proof}

\begin{corollary}\label{cor:epsilon-sections}
Let $G$ be a spinal Hanoi towers group, and let $g\in \St_G(1)$ with $\psi(g)=(g_1,\ldots,g_d)$.
Then
\[
\varepsilon(g)=\sum_{j=1}^d \varepsilon(g_j).
\]
\end{corollary}

\begin{proof}
Let $\gamma=a_{i_1}^{r_1}\cdots a_{i_k}^{r_k}$ be a syllable form evaluating to $g$. For each $i\in\{1,\ldots,d\}$, we have
already observed that
\[
e_i(\gamma)=\sum_{j=1}^d e_i(\gamma|_j).
\]
By \Cref{epsilonmap}, the maps $e_i$ and $\varepsilon$ are well defined on $G$, independently of the chosen syllable expression. Therefore the previous coordinatewise identities give $\varepsilon(g)=\sum_{j=1}^d \varepsilon(g_j)$.
\end{proof}

% \fc{The results established in this section provide the key ingredients for the proof of \Cref{thmA}. Indeed \Cref{fractandtrans} proves point $(i)$; \Cref{remarkamenabke}, \Cref{alg:word-problem} together with \Cref{cornucleo} prove point $(ii)$; and \Cref{epsilonmap} proves point $(iii)$.}

\section{Branch structure and the first level stabilizer}
\label{sectbranch}

In this section we study the first level stabilizer of a spinal Hanoi towers group. More precisely, we look for conditions describing which tuples
\[
(g_1,\ldots,g_d)\in G\times \stackrel{d}{\cdots}\times G
\]
belong to $\psi(\St_G(1))$.
Indeed, a more precise description of $\St_G(1)$ can help determine whether $G$ admits a larger or more natural branching subgroup, and in particular whether $G$ is layered.

The following key lemma provides a  consequence of strong fractality.

\begin{lemma}\label{strongfractimpliesbranch}
	Let $G$ be a strongly fractal spinal Hanoi towers group, with $d\geq 4$ and
	$\alt(d)\leq P$. Suppose that, for every $i\in\{1,\ldots,d\}$, there exists $x_i\in\St_G(1)$ such that:
	\begin{enumerate}
		\item some section of $x_i$ is equal to %$a_i^{\pm1}$;
        $a_{i}^{t}$, for some $t$ coprime with $o(a_{i})$;
        \item all the sections of $x_{i}$ commute;
		\item the set of coordinates at which $x_i$ has a non-trivial
		section, that is,\\  $S_i=\{t\in X:(x_i)|_t\neq 1\}$ satisfies  $2|S_i|-1\leq d$.
	\end{enumerate}
	Then $G$ is regular branch over $G'$.
\end{lemma}

\begin{proof}
	It suffices to show that for every $i,j\in \{1,\ldots,d\}$, the group $G$ contains some element of the form 
	\[
	(1,\ldots,1,[a_i,a_j],1,\ldots,1)
	\in \psi\bigl(\St_{G'}(1)\bigr).
	\]
    Indeed, the result then  follows from the fact that $G$ is level transitive  (since $\alt(d)\leq P$) and strongly fractal.
	To that end,  fix $i,j\in \{1,\ldots,d\}$. For $k\in \{i,j\}$, replacing $x_k$ by the power $x_k^{s}$ with $ts\equiv 1\pmod{o(a_k)}$ if necessary, we may choose
	$q_k\in S_k$ as in $(i)$, such that $(x_k)|_{q_k}=a_k$. Since, by hypothesis (iii), for $d\geq 4$ we have $|S_i|+|S_j|-1\leq d$ and $|S_j|\leq d-2$, the $(d-2)$-transitivity of $\alt(d)$ yields $\tau\in P$ such that
	\[
	q_j^\tau=q_i
	\qquad\text{and}\qquad
	S_j^\tau\cap S_i=\{q_i\}.
	\]
	Choose $z_0\in G$ with root permutation $\tau$. By strong
	fractality, we may multiply $z_0$ by an element of
	$\St_G(1)$ and obtain $z\in G$ such that
	\[
	z|_{q_j}=1.
	\]
	Consequently, $x_j^z$ has support $S_j^\tau$ and
	\[
	(x_j^z)|_{q_i}=a_j.
	\]
	Therefore, by $(ii)$,
	\[
	\psi([x_i,x_j^z])=(1,\ldots,1, \underset{q_i}{[a_i,a_j]}, 1,\ldots,1),
	\]
	concluding the proof.
\end{proof}

We begin with a necessary condition coming from the abelianization of $G$. This condition will later serve as the main obstruction to deciding whether a tuple in $G\times \stackrel{d}{\cdots}\times G$ belongs to $\psi(\St_G(1))$.
By \Cref{epsilonmap}, the map
\[
\varepsilon:G\longrightarrow \prod_{i=1}^d \mathbb Z/o(a_i)\mathbb Z
\]
induces an isomorphism
\[
\overline{\varepsilon}:G/G'\longrightarrow
\prod_{i=1}^d \mathbb Z/o(a_i)\mathbb Z.
\]

On the other hand, the action of $G$ on the first level gives a surjective homomorphism onto $P=\langle\sigma_1,\ldots,\sigma_d\rangle$
\begin{align}\label{therhomap}
    \rho:G&\longrightarrow P\\
    g=\psi^{-1}((g_1,\ldots,g_d)\tau)&\longmapsto \tau
\end{align}
Since $\rho(G')\leq P'$, this induces a homomorphism
\begin{align*}
    \overline{\rho}:G/G'&\longrightarrow P/P'\\
    gG'&\longmapsto \rho(g)P'
\end{align*}
Using the isomorphism $\overline{\varepsilon}$, we obtain a homomorphism
\[
\lambda:
\prod_{i=1}^d \mathbb Z/o(a_i)\mathbb Z
\longrightarrow P/P'
\]
defined by
\[
\lambda(v)=\overline{\rho}\bigl(\overline{\varepsilon}^{-1}(v)\bigr).
\]
Thus $\lambda$ measures the obstruction, visible in the abelianization of $G$, for an exponent vector to come from an element acting trivially on the first level. The following proposition makes this observation precise.
\begin{proposition}\label{prop:necessary-condition-first-stabilizer}
Let $g\in \St_G(1)$ and write $\psi(g)=(g_1,\ldots,g_d)$.
Then $\sum_{j=1}^d \varepsilon(g_j)\in \ker(\lambda)$.
Consequently,
\[
\psi(\St_G(1))
\subseteq
\left\{
(g_1,\ldots,g_d)\in G\times \stackrel{d}{\cdots}\times G
\ \middle|\
\sum_{j=1}^d \varepsilon(g_j)\in \ker(\lambda)
\right\}.
\]
\end{proposition}

\begin{proof}
Let $g\in \St_G(1)$ and write $\psi(g)=(g_1,\ldots,g_d)$. Since $g$ fixes the
first level, we have $\rho(g)=1$, and therefore $\overline{\rho}(gG')=P'$. Equivalently, $\lambda(\varepsilon(g))=P'$. By \Cref{cor:epsilon-sections} we have $\varepsilon(g)=\sum_{j=1}^d \varepsilon(g_j)$. Hence
\[
\lambda\left(\sum_{j=1}^d \varepsilon(g_j)\right)=P',
\]
which proves that
\[
\sum_{j=1}^d \varepsilon(g_j)\in \ker(\lambda).
\]
The stated inclusion follows immediately.
\end{proof}

We now introduce a key subgroup that will be used throughout this section to describe $\St_G(1)$. Roughly speaking, this subgroup allows us to compare different coordinates in the wreath recursion.
\begin{definition}\label{defofJ}
Let $G$ be a spinal Hanoi towers group. We denote by $J$ the subgroup generated by all
elements of $\St(1)$ whose wreath recursion is trivial in all but two
coordinates, where the two exceptional coordinates are occupied by some element
$h\in G$ and its inverse $h^{-1}$. More precisely,
\[
J =
\left\langle
g\in \St(1)
\ \middle|\
\psi(g)
=
(1,\ldots,1,h,1,\ldots,1,h^{-1},1,\ldots,1)
\text{ for some } h\in G
\right\rangle .
\]
Here the two entries $h$ and $h^{-1}$ are allowed to occur in any two distinct
coordinates.
\end{definition}

Note that although the elements used to define $J$ are described by a wreath recursion whose entries lie in $G$, they are a priori elements of $\Aut(\T_d)$, not necessarily elements of $G$.

Also note that if $J$ is contained in $G$ then $G$ is strongly fractal.
\begin{lemma}\label{lemma:J-contains-derived-sections}
Let $G$ be a spinal Hanoi towers group, and let $J$ be as in
\Cref{defofJ}. If $J\leq G$, then
\begin{enumerate}
    \item $G$ is strongly fractal. 
    \item $G'\times \stackrel{d}{\cdots}\times G' \leq \psi(J')$ and $J$ is a normal subgroup of $G$. In particular, if $G$ is level transitive, $G$ is regular branch over $G'$.
\end{enumerate}
\end{lemma}

\begin{proof}
	The first item follows directly from the definition of $J$. For the second item, let
	$p\neq q$ and $h\in G$, and write $\delta_{p,q}(h)$ for an element whose first level
	decomposition has $h$ in the $p$-th coordinate, $h^{-1}$ in the $q$-th coordinate, and
	trivial entries elsewhere. By definition, all such elements belong to $J$.
	
	We first show that $\psi(J')$ contains a copy of $G'$ in each coordinate. Fix $q$ and
	choose $p,r$ distinct from $q$ and from each other. For $x,y\in G$, we have
	\[
	[\delta_{p,q}(x^{-1}),\delta_{q,r}(y)]\in J'.
	\]
	The two elements have overlapping support only in the $q$-th coordinate, where their
	sections are $x$ and $y$, respectively. Hence
	\[
	\psi\bigl([\delta_{p,q}(x^{-1}),\delta_{q,r}(y)]\bigr)
	=
	(1,\ldots,1,[x,y],1,\ldots,1),
	\]
	where $[x,y]$ appears in the $q$-th coordinate. Since $x$ and $y$ are arbitrary, this
	shows that $\psi(J')$ contains the subgroup supported in the $q$-th coordinate with
	entries in $G'$. Varying $q$, we obtain
	\[
	G'\times \stackrel{d}{\cdots}\times G'\leq \psi(J').
	\]
	
	It remains to prove normality. Let $z\in G$, and write
	\[
	\psi(z)=(z_1,\ldots,z_d)\tau.
	\]
	It is enough to conjugate the generators $\delta_{p,q}(h)$ of $J$. The conjugate
	$\delta_{p,q}(h)^z$ again belongs to $\St(1)$, and its first level decomposition is
	trivial outside the two coordinates $\tau^{-1}(p)$ and $\tau^{-1}(q)$. In these two
	coordinates the entries are conjugates of $h$ and $h^{-1}$ by sections of $z$. Thus
	$\psi(\delta_{p,q}(h)^z)$ differs from
	$\psi(\delta_{\tau^{-1}(p),\tau^{-1}(q)}(h))$ by an element of
	$G'\times \stackrel{d}{\cdots}\times G'$, which is contained in $\psi(J')\leq \psi(J)$.
	Therefore $\delta_{p,q}(h)^z\in J$, and so $J$ is normal in $G$.
\end{proof}

We note that the wreath recursion gives the following natural homomorphism
\begin{align}\label{natural-wreath-recursion}
    \Psi:\psi(G)&\longrightarrow  \left(\left(\prod_{i=1}^d \mathbb Z/o(a_i)\mathbb Z\right)\times \stackrel{d}{\cdots}\times \left(\prod_{i=1}^d \mathbb Z/o(a_i)\mathbb Z\right)\right)\rtimes P,\\
    (g_1,\ldots,g_d)\tau &\longmapsto \bigl(\varepsilon(g_1),\ldots,\varepsilon(g_d)\bigr)\tau\nonumber
\end{align} 
whose kernel is contained in $G'\times \stackrel{d}{\cdots}\times G'$, in other words, $\ker\Psi\leq G'\times \stackrel{d}{\cdots}\times G'$. Hence, if $G$ is regular branch over $G'$ quotienting by $G'\times \stackrel{d}{\cdots}\times G'$ gives the following embedding
\begin{align}\label{Psihook}
    \psi(G)/\left(G'\times \stackrel{d}{\cdots}\times G'\right)\hookrightarrow  \left(\left(\prod_{i=1}^d \mathbb Z/o(a_i)\mathbb Z\right)\times \stackrel{d}{\cdots}\times \left(\prod_{i=1}^d \mathbb Z/o(a_i)\mathbb Z\right)\right)\rtimes P.
\end{align}
If, moreover, the subgroup $J$ is contained in $G$, then we can quotient further by the relations coming from $J$. These relations identify tuples which differ by moving an element from one coordinate to another with inverse contribution. That is, for every $(g_1,\ldots,g_d)\in G\times \stackrel{d}{\cdots}\times G$,
\[
(g_1,\ldots,g_d)\equiv  (1,\ldots,1,g_{\sigma(1)}g_{\sigma(2)}\cdots g_{\sigma(d)},1,\ldots,1)\pmod{\psi(J)},
\]
for every $\sigma\in \Sym(d)$. Therefore, modulo $J$, the individual abelianized sections $\varepsilon(g_1),\ldots,\varepsilon(g_d)$ are no longer seen separately; only their sum remains. Thus, the map in \Cref{natural-wreath-recursion} induces:
%\begin{align}\label{GJhook}
%    G/J\hookrightarrow \left(\prod_{i=1}^d\mathbb Z/o(a_i)\mathbb Z \right)\times P.
%\end{align}
\begin{align}\label{GJhookexplicit}
    (\varepsilon,\rho):G/J&\longrightarrow \left(\prod_{i=1}^d\mathbb Z/o(a_i)\mathbb Z \right)\times P\\
    gJ&\longmapsto (\varepsilon(g),\rho(g)).\nonumber
\end{align}

This map is well-defined because of the following: since every generator of $J$ lies in $\St_G(1)$ and has the form
\[
\psi(x)=(1,\ldots,1,h,1,\ldots,1,h^{-1},1,\ldots,1),
\]
we have $\rho(x)=1$. Moreover, by \Cref{cor:epsilon-sections},
\[
\varepsilon(x)=\varepsilon(h)+\varepsilon(h^{-1})=0.
\]
Hence $J\leq \ker(\varepsilon)\cap\ker\rho$, and therefore both $\varepsilon$ and $\rho$
factor through $G/J$. 

As a consequence of this discussion, we have the following two lemmas whose proof are straightforward.

\begin{lemma}\label{lem:JeqGprimeSt}
Let $G$ be a spinal Hanoi towers group with $J\leq G$. Then
\[
J=G'\cap \St_G(1).
\]
In particular, the map $(\varepsilon,\rho)$ defined above is a well-defined injective homomorphism.
\end{lemma}

\begin{lemma}Let $G$ be a spinal Hanoi towers group. If $G$ is regular branch over $G'$, then the quotient $\psi(G)/(G'\times \stackrel{d}{\cdots}\times G')$ is isomorphic to a subgroup of 
\[
 \left(\left(\prod_{i=1}^d \mathbb Z/o(a_i)\mathbb Z\right)\times \stackrel{d}{\cdots}\times \left(\prod_{i=1}^d \mathbb Z/o(a_i)\mathbb Z\right)\right)\rtimes P.
\]
If, moreover, the subgroup $J$ is contained in $G$, then
the quotient $G/J$ is isomorphic to a subgroup of
\[
 \left(\prod_{i=1}^d \mathbb Z/o(a_i)\mathbb Z\right)\times P.
\]
\end{lemma}

We are now ready to prove the main result of this section.

\begin{proposition}\label{prop:sufficient-condition-first-stabilizer}
Let $G$ be a spinal Hanoi towers group. If $J\leq G$ then 
\[
\psi(\St_G(1))
=
\left\{
(g_1,\ldots,g_d)\in G\times \stackrel{d}{\cdots}\times G \ \middle| \sum_{j=1}^d \varepsilon(g_j)\in \ker(\lambda)
\right\}.
\]
\end{proposition}

\begin{proof}
First we prove that $\varepsilon(\St_G(1))=\ker(\lambda)$.
The inclusion $\varepsilon(\St_G(1))\subseteq \ker(\lambda)$ follows from \Cref{prop:necessary-condition-first-stabilizer}.

Conversely, let $v\in\ker(\lambda)$. Since $\varepsilon:G\to \prod_{i=1}^d \mathbb Z/o(a_i)\mathbb Z$
is surjective, choose $g\in G$ such that $\varepsilon(g)=v$. The condition $v\in\ker(\lambda)$ implies $\rho(g)\in P'$. Since $\rho:G\to P$ also surjective, we have $\rho(G')=P'$. Hence there exists $c\in G'$ such that $\rho(c)=\rho(g)$. Therefore $\rho(gc^{-1})=1$, so $gc^{-1}\in \St_G(1)$. Moreover, since $G'=\ker(\varepsilon)$, we have $\varepsilon(gc^{-1})=\varepsilon(g)=v$.
Thus $v\in \varepsilon(\St_G(1))$, and therefore $\varepsilon(\St_G(1))=\ker(\lambda)$. 

To conclude, recall that the map $(\varepsilon,\rho)$ in \Cref{GJhookexplicit} is injective by \Cref{lem:JeqGprimeSt}, and so also its restriction
\[
(\varepsilon,\rho)|_{\St_G(1)/J}:\St_G(1)/J\longrightarrow \left(\prod_{i=1}^d\mathbb Z/o(a_i)\mathbb Z \right)\times P.
\]
By also corestricting the codomain to the image and using that $\varepsilon(\St_G(1)/J)\cong\varepsilon(\St_G(1))=\ker(\lambda)$  we get the following isomorphism
\[
(\varepsilon,\rho)|_{\St_G(1)/J}:\St_G(1)/J\longrightarrow \ker(\lambda).
\]
Hence
\[
\psi(\St_G(1))/\psi(J)
=
\left\{
(g_1,\ldots,g_d)\psi(J)
\ \middle|\
(g_1,\ldots,g_d)\in G\times \stackrel{d}{\cdots}\times G,\ 
\sum_{j=1}^d \varepsilon(g_j)\in \ker(\lambda)
\right\}.
\]
Finally, since $\psi(J)\leq\psi(\St_G(1))\leq G\times \stackrel{d}{\cdots}\times G$ and $\varepsilon(J)=\{0\}$ it follows that
\[
\psi(\St_G(1))
=
\left\{
(g_1,\ldots,g_d) \ \middle| (g_1,\ldots,g_d)\in G\times \stackrel{d}{\cdots}\times G,\sum_{j=1}^d \varepsilon(g_j)\in \ker(\lambda)
\right\}.
\]

\end{proof}
\begin{corollary}\label{perfectcase}
	Let $G$ be a spinal Hanoi towers group such that $J\leq G$. If $P$ is perfect, then $G$ is layered, in other words
	\[
	\psi(\St_G(1))=G\times \stackrel{d}{\cdots}\times G.
	\]
\end{corollary}

\begin{proof}
	Since $P$ is perfect, we have $P/P'=1$. Hence the map $\lambda$ is trivial.
\end{proof}

\begin{corollary}\ml{\label{cor:symcase}}
	Let $G$ be a spinal Hanoi towers group such that $J\leq G$. Assume that $P=\Sym(d)$. Then
	\[
	\psi(\St_G(1))
	=
	\left\{
	(g_1,\ldots,g_d)\in G\times \stackrel{d}{\cdots}\times G
	\ \middle|\
	\sum_{\sigma_i\notin \alt(d)}\sum_{j=1}^d e_i(g_j)\equiv 0 \pmod 2
	\right\}.
	\]
\end{corollary}

\begin{proof}
	Since $P=\Sym(d)$, we have $P'=\alt(d)$ and $P/P'\cong \mathbb Z/2\mathbb Z$. Under this identification, the map $\lambda$ records the parity of the root permutation. Hence
	\[
	\lambda(\varepsilon(a_i))=
	\begin{cases}
		0, & \text{if } \sigma_i\in \alt(d),\\
		1, & \text{if } \sigma_i\notin \alt(d).
	\end{cases}
	\]
	Therefore an element $v=(r_1,\ldots,r_d)\in \prod_{i=1}^d \mathbb Z/o(a_i)\mathbb Z$ belongs to $\ker(\lambda)$ if and only if
	\[
	\sum_{\sigma_i\notin \alt(d)} r_i\equiv 0 \pmod 2.
	\]
\end{proof}

Using the same reasoning one can get the slightly more general result.
\begin{corollary}\label{cor:rem}
	Let $G$ be a spinal Hanoi towers group such that $J\leq G$. Assume that $P/P'\cong \mathbb Z/p\mathbb Z$ for some prime $p$. For any isomorphism $P/P'\to \mathbb Z/p\mathbb Z$ where $\sigma_iP'\mapsto c_i$, we have
		\[
	\psi(\St_G(1))
	=
	\left\{
	(g_1,\ldots,g_d)\in G\times \stackrel{d}{\cdots}\times G
	\ \middle|\
	\sum_{i=1}^d c_i\sum_{j=1}^d e_i(g_j)\equiv0 \pmod p
	\right\}.
	\]
\end{corollary}
For the remainder of this section, we give a sufficient condition ensuring that $J$ is contained in $G$. We do this by describing the abelianized image of $J$ under the map $\Psi$ from \Cref{natural-wreath-recursion}.

Consider the homomorphism
\begin{align} \label{summap}
    \Sigma:\left(\prod_{i=1}^d \mathbb Z/o(a_i)\mathbb Z\right)^d
    &\longrightarrow \prod_{i=1}^d \mathbb Z/o(a_i)\mathbb Z,\\
    (v_1,\ldots,v_d)&\longmapsto v_1+\cdots+v_d.\nonumber
\end{align}
Let $D:=\ker \Sigma$. Note that $D$ can be equivalently defined as the group generated by the elements
$d_{p,q}(v):=(0,\ldots,0,v,0,\ldots,0,-v,0,\ldots,0)$,
where $v\in \prod_{i=1}^d \mathbb Z/o(a_i)\mathbb Z$ and the two non-zero entries occur in distinct coordinates $p\neq q$. Thus $D$ is precisely the abelianized shadow of the generators of $J$.

\begin{lemma}\label{Transfer}
    Let $G$ be a level transitive spinal Hanoi towers group. Then $J\leq G$ if and only if $G$ is regular branch over $G'$ and $D\leq \Psi(\St_G(1))$. 
\end{lemma}

\begin{proof}
    Assume first that $J\leq G$. By \Cref{lemma:J-contains-derived-sections}, the group $G$ is regular branch over $G'$. Moreover, if $\delta_{p,q}(h)$ is one of the generators of $J$, then
    $\Psi(\delta_{p,q}(h))=d_{p,q}(\varepsilon(h))$.
    Since 
    $$\varepsilon:G\to \prod_{i=1}^d \mathbb Z/o(a_i)\mathbb Z$$
    is surjective, every element $v$ of this product is of the form $\varepsilon(h)$ for some $h\in G$. Hence every generator $d_{p,q}(v)$ of $D$ belongs to $\Psi(\St_G(1))$, and therefore $D\leq \Psi(\St_G(1))$.

    Conversely, assume that $G$ is regular branch over $G'$ and that $D\leq \Psi(\St_G(1))$. Let $\delta_{p,q}(h)$ be a generator of $J$. Since $d_{p,q}(\varepsilon(h))\in D$, there exists $x\in \St_G(1)$ such that
    $\Psi(x)=d_{p,q}(\varepsilon(h))$.
    The elements $x$ and $\delta_{p,q}(h)$ therefore have the same image under $\Psi$. Equivalently, all sections of $x\delta_{p,q}(h)^{-1}$ lie in $G'$. Thus
    $\psi(x\delta_{p,q}(h)^{-1})\in G'\times \cdots \times G'$.

    Since $G$ is regular branch over $G'$, there exists $y\in G'$ such that
    $\psi(y)=\psi(x\delta_{p,q}(h)^{-1})$.
    By injectivity of the wreath recursion, $y=x\delta_{p,q}(h)^{-1}$. Hence
    $\delta_{p,q}(h)=y^{-1}x\in G$.
    Since this holds for every generator $\delta_{p,q}(h)$ of $J$, we conclude that $J\leq G$.
\end{proof}

This criterion is useful because the condition $D\leq \Psi(\St_G(1))$ can be checked in a finite quotient. Indeed, $D$ is generated by the elements $d_{p,q}(\varepsilon(a_i))$, with $1\leq i\leq d$ and $p\neq q$, since the vectors $\varepsilon(a_1),\ldots,\varepsilon(a_d)$ generate $\prod_{i=1}^d \mathbb Z/o(a_i)\mathbb Z$. The next proposition gives a practical way to verify this condition. 
\begin{proposition}\label{prop:two-transitive-transfer}
    Assume that $G$ is regular branch over $G'$ and that $P$ is 2-transitive on $\{1,\ldots,d\}$. Suppose that, for every $i\in\{1,\ldots,d\}$, there exist distinct coordinates $p_i,q_i\in\{1,\ldots,d\}$ and an element $x_i\in \St_G(1)$ such that
    $$
    \psi(x_i)\equiv(1,\ldots,1,a_i,1,\ldots,1,a_i^{-1},1,\ldots,1)
    \pmod{(G'\times \dots \times G')},
    $$
    where $a_i$ occurs in the $p_i$-th coordinate and $a_i^{-1}$ occurs in the $q_i$-th coordinate. Then $J\leq G$.
\end{proposition}
\begin{proof}
    By \Cref{Transfer}, it is enough to prove that $D\leq \Psi(\St_G(1))$.

    Fix $i\in\{1,\ldots,d\}$. By assumption, $\Psi(\St_G(1))$ contains the element $d_{p_i,q_i}(\varepsilon(a_i))$. Let $p\neq q$ be any ordered pair of distinct coordinates. Since $P$ is 2-transitive, there exists $\tau\in P$ such that $\tau(p_i)=p$ and $\tau(q_i)=q$. Choose $z\in G$ whose root permutation is $\tau$. Then $x_i^z$ still belongs to $\St_G(1)$, and conjugation by $z$ permutes the abelianized sections according to $\tau$. Hence $\Psi(\St_G(1))$ also contains $d_{p,q}(\varepsilon(a_i))$.

    Since $p\neq q$ was arbitrary, we have $d_{p,q}(\varepsilon(a_i))\in \Psi(\St_G(1))$ for every ordered pair $p\neq q$. This holds for every $i$, and the elements $\varepsilon(a_1),\ldots,\varepsilon(a_d)$ generate $\prod_{j=1}^d \mathbb Z/o(a_j)\mathbb Z$. Therefore the generators $d_{p,q}(v)$ of $D$ all belong to $\Psi(\St_G(1))$, and so $D\leq \Psi(\St_G(1))$.
\end{proof}

This criterion can be formulated as the following finite algorithm.

\begin{algorithm}[H]
\caption{Checking the criterion for $J\leq G$}\label{alg:Jcriterion}
\begin{algorithmic}[1]
\Require The permutations $\sigma_1,\ldots,\sigma_d$, the orders $o(a_1),\ldots,o(a_d)$, and the assumption that $G$ is regular branch over $G'$.
\Ensure Whether the criterion of \Cref{Transfer} holds.

\State Let $P=\langle \sigma_1,\ldots,\sigma_d\rangle$.
\State Work in the finite group $\left(\prod_{i=1}^d \mathbb Z/o(a_i)\mathbb Z\right)^d\rtimes P$.
\State For each $i$, define $\widehat a_i$ to be the element with root permutation $\sigma_i$ and with abelianized sections
\[
(0,\ldots,0,\varepsilon(a_i),0,\ldots,0),
\]
where $\varepsilon(a_i)$ occurs in the $i$-th coordinate.
\State Compute $H=\langle \widehat a_1,\ldots,\widehat a_d\rangle$.
\State Let $S$ be the set of elements of $H$ with trivial root permutation.
\State For every $i\in\{1,\ldots,d\}$ and every pair $p\neq q$, check whether $d_{p,q}(\varepsilon(a_i))$ belongs to $S$.
\If{all these elements belong to $S$}
    \State Return true: by \Cref{Transfer}, $J\leq G$.
\Else
    \State Return false: the criterion in \Cref{Transfer} does not hold.
\EndIf
\end{algorithmic}
\end{algorithm}

\section{Just infiniteness}\label{secjustinf}
In this section we prove that level transitive spinal Hanoi towers groups satisfying $J\leq G$ are just infinite. We recall the following.

\begin{definition}
Let $H$ be a group and $\mathfrak{X}$ a property of groups. We say that $H$ is \emph{just} $\mathfrak{X}$ if $H$ has the property $\mathfrak{X}$ and every proper quotient of $H$ does not have the property $\mathfrak{X}$.
In particular, a group $G$ is \textit{just infinite}  if $G$ is infinite and if all of its proper quotients are finite.
\end{definition}

We will use the following criterion of Grigorchuk.

\begin{theorem}[Thm. 4, \cite{JI grigorchuk}]
A branch group $G$ is just infinite if and only if, for each $m \geq 1$, the index $|\Rist_G(m):\Rist_G(m)'|$ is finite.
\end{theorem}
We will apply this criterion by giving an explicit description of the rigid stabilizers at every level in terms of a fixed finite-index subgroup of $G$.

Consider
\[
B=\varepsilon^{-1}(\ker\lambda)
=\rho^{-1}(P')
=
\bigl\{\, g\in G \mid \rho(g)\in P' \,\bigr\}.
\]
Then, if $J \le G$, the description of $\St_G(1)$ in
\Cref{prop:sufficient-condition-first-stabilizer} yields that an element of the form
$(1,\ldots,1,h,1,\ldots,1)$ belongs to $\psi(\St_G(1))$ if and only if
$\varepsilon(h)\in\ker\lambda$, equivalently if and only if $h\in B$. Hence
\[
\psi(\Rist_G(1))
=
B\times\stackrel{d}{\cdots}\times B.
\]
Iterating the same argument,   
we obtain
\[
\psi_m(\Rist_G(  m))
=
B\times\stackrel{d^m}{\cdots}\times B.
\]
We collect this information in the following proposition.

\begin{proposition}\ml{\label{prop:rist-explicit}}
	Let $G$ be a spinal Hanoi towers group and suppose that $J \le G$. Then, for every
	$m\geq 1$,
	\[
	\psi_m(\Rist_G(m))
	=
	B\times\stackrel{d^m}{\cdots}\times B.
	\]
\end{proposition}

Thus, it suffices to prove that $|B:B'|<\infty$. To this end, note first that
\[
J\leq G'\leq B.
\]
Indeed, $J\leq G'$ because every generator of $J$ has total exponent zero, while
$G'=\ker\varepsilon\leq \varepsilon^{-1}(\ker\lambda)=B$. Since $B'\leq G'\leq B$, we have
\[
|B:B'|=|B:G'|\,|G':B'|.
\]
Consequently, in order to prove that $|B:B'|<\infty$, it is enough to prove that both $|B:G'|$ and $|G':B'|$ are finite. Note that the first index is finite as $G/G'$ is finite by \Cref{epsilonmap}.
Therefore it remains to prove that $|G':B'|<\infty$.

\begin{lemma}\label{lemma:Gprime-mod-Bprime-finite}
	Let $G$ be a spinal Hanoi towers group and suppose that $J\leq G$. Then
	\[
	|G':B'|<\infty .
	\]
\end{lemma}

\begin{proof}
	Since $B'\leq G'$ and $J\leq G'$, we have
	\[
	|G':B'|=|G':JB'|\,|JB':B'|.
	\]
	Thus it suffices to prove that both factors are finite. 

	Now $\rho(G')=P'$, and the kernel of $\rho|_{G'}$ is
	$G'\cap \St_G(1)=J$, where the last equality follows from \Cref{lem:JeqGprimeSt}. Thus
	$G'/J\cong P'$,
	so $G'/J$ is finite. Since $J\leq JB'\leq G'$, it follows that $G'/JB'$ is finite.
	
	It remains to prove that $JB'/B'$ is finite. For $p\neq q$, consider
	\[
	\theta_{p,q}\colon G\longrightarrow B/B',
	\qquad
	h\longmapsto \delta_{p,q}(h)B'.
	\]
	This is a homomorphism, because
	\[
	\delta_{p,q}(x)\delta_{p,q}(y)\delta_{p,q}(xy)^{-1}\in B'
	\]
	by \Cref{lemma:J-contains-derived-sections}. Moreover, if $h\in G'$, then
	$\delta_{p,q}(h)\in B'$, again by \Cref{lemma:J-contains-derived-sections}. Hence
	$G'\leq \ker\theta_{p,q}$, so $\theta_{p,q}$ factors through the finite group $G/G'$.
	Therefore each image $\operatorname{im}\theta_{p,q}$ is finite.
	
	Finally, $J$ is generated by the elements $\delta_{p,q}(h)$ with $p\neq q$ and $h\in G$.
	Since there are only finitely many pairs $p\neq q$, the image of $J$ in the abelian group
	$B/B'$ is generated by finitely many finite subgroups. Hence $JB'/B'$ is finite. Therefore
	$|G':B'|<\infty$.
\end{proof}
\begin{theorem}\label{thm:justinfinite}
Let $G$ be a level transitive spinal Hanoi towers group and suppose that $J\leq G$. Then $G$ is just infinite.
\end{theorem}

\begin{proof}
By \Cref{lemma:J-contains-derived-sections}(ii), we have $G'\times \stackrel{d}{\cdots}\times G'\leq\psi(J')\leq\psi(G')$; since $|G:G'|<\infty$ by \Cref{epsilonmap} and $G$ is level transitive, $G$ is regular branch over $G'$. In particular, $G$ is a branch group. Therefore, by Grigorchuk's
criterion, it is enough to prove that
\[
|\Rist_G(m):\Rist_G(m)'|<\infty
\]
for every $m\geq 1$.

By the description of the rigid stabilizers obtained above,
\[
\psi_m(\Rist_G( m))
=
B\times\stackrel{d^m}{\cdots}\times B.
\]
Hence
\[
\psi_m(\Rist_G( m)')
=
B'\times\stackrel{d^m}{\cdots}\times B',
\]
and therefore
\[
|\Rist_G( m):\Rist_G( m)'|
=
|B:B'|^{d^m}.
\]
By the previous lemmas, $|B:B'|<\infty$. Thus
$|\Rist_G( m):\Rist_G( m)'|<\infty$ for every $m\geq 1$, and Grigorchuk's criterion
implies that $G$ is just infinite.
\end{proof}

The results proved above show that the containment $J\leq G$ is a sufficient condition for a level transitive spinal Hanoi towers group $G$ to be just infinite. The classical Hanoi towers group $\mathcal{H}^{(3)}$ is known not to be just infinite \cite{Rachel} and, to the best of our knowledge, it is currently the only known level transitive spinal Hanoi towers group with this property. Moreover, $J\nleq \mathcal{H}^{(3)}$. This motivates the following question.

\begin{question}\label{q:general}
	Let $G$ be a level transitive spinal Hanoi towers group. Does
	\[
	J\nleq G
	\qquad\Longrightarrow\qquad
	G\text{ is not just infinite}?
	\]
	Equivalently, is the failure of $J\leq G$ always accompanied by a proper infinite quotient of $G$?
\end{question}
A natural test case is the $5$-ary Hanoihedral group $\Dfive$, introduced by Garrido and \v{S}uni\'c~\cite{GarridoSunic2025}. It is generated by $a_1,\ldots,a_5$, where
\begin{align*}
	\psi(a_1)&=(a_1,1,1,1,1)(2 \ 5)(3 \ 4),&
	\psi(a_2)&=(1,a_2,1,1,1)(1\ 3)(4 \ 5),\\
	\psi(a_3)&=(1,1,a_3,1,1)(1 \ 5)(2 \ 4),&
	\psi(a_4)&=(1,1,1,a_4,1)(1 \ 2)(3 \ 5),\\
	\psi(a_5)&=(1,1,1,1,a_5)(1 \ 4)(2 \ 3).
\end{align*}
Garrido and \v{S}uni\'c proved that $\Dfive$ is fractal, contracting, and regular branch over its derived subgroup. Moreover, they proved the following:
\begin{enumerate}
    \item $\Dfive/\Dfive'\cong \mathbb Z/2\mathbb Z\times\stackrel{5}{\cdots}\times\mathbb Z/2\mathbb Z$
\item For $n \geq 1$, $\Rist_{\Dfive}(n)\cong \Dfive'\times\stackrel{5^n}{\cdots}\times\Dfive'$.
\end{enumerate}

They also prove that the rigid kernel of $\Dfive$ is an infinite
elementary abelian $2$-group, isomorphic to a countable Cartesian
product of copies of $\mathbb{Z}/2\mathbb{Z}$. This does not, however,
determine whether $\Dfive$ is abstractly just infinite.

The equations describing the first-level stabilizer in
\cite[Proposition~3.7(iii)]{GarridoSunic2025} also give a direct
obstruction to the containment $J\leq \Dfive$. Let
$g\in \St_{\Dfive}(1)$ and write $\psi(g)=(g_1,\ldots,g_5)$.

Throughout this paragraph, the subscripts of the exponent maps $e_k$
are read modulo $5$, with representatives in $\{1,\ldots,5\}$. Then,
for every $r\in\{1,\ldots,5\}$,
\[
\sum_{i=1}^{5} e_{r+i-1}(g_i)\equiv 0\pmod 2,
\qquad
\sum_{i=1}^{5} e_{r-i+1}(g_i)\equiv 0\pmod 2.
\]

Fix distinct $p,q\in\{1,\ldots,5\}$ and
$j\in\{1,\ldots,5\}$. Since $a_j$ is an involution, the element
$\delta_{p,q}(a_j)$ has section $a_j$ in the $p$-th and $q$-th
coordinates and trivial sections elsewhere. Choose
$r\in\{1,\ldots,5\}$ such that
\[
r+p-1\equiv j\pmod 5,
\]
or equivalently, $r\equiv j-p+1\pmod 5$. In the first family of
equations, the $p$-th summand is then
\[
e_{r+p-1}(a_j)=e_j(a_j)=1.
\]
The $q$-th summand is zero, since $q\neq p$ implies
$r+q-1\not\equiv j\pmod 5$, and all the remaining summands are zero
because the corresponding sections are trivial. Hence the associated
parity condition is not satisfied, and therefore
\[
\delta_{p,q}(a_j)\notin\Dfive.
\]
It follows that $J\nleq\Dfive$. Thus, $\Dfive$ is a particularly good
test case for Question~\ref{q:general}.

If $\Dfive$ is not just infinite, it provides a second natural
Hanoi-type example, beyond $\mathcal{H}^{(3)}$, in which the failure
of $J\leq G$ is accompanied by the failure of just infiniteness. On
the other hand, if $\Dfive$ is just infinite, then
Question~\ref{q:general} has a negative answer.

% \fc{The results of \Cref{sectbranch} and \Cref{secjustinf} prove \Cref{thmB}. In particular points $(i)-(ii)$
%  are proved in \Cref{lemma:J-contains-derived-sections}; point $(iii)$ is proved in \Cref{prop:sufficient-condition-first-stabilizer}; point $(iv)$ is proved in \Cref{prop:rist-explicit}; point $(v)$ is proved in \Cref{thm:justinfinite} with \Cref{alg:Jcriterion}. Finally the proof is concluded with \Cref{Transfer}}

\section{Examples: spinal Hanoi towers groups of type $(d,m)$}\label{sec:examples}
In this section, we illustrate the strength of the machinery developed in the previous sections by applying it to a natural generalization of the Hanoi towers group $\mathcal{H}^{(3)}$, which we call the Hanoi towers groups of type $(d,m)$. These groups fall into the broader family of spinal Hanoi towers groups.

\begin{definition}\label{def:typem}
Let $d\geq 3$ and $2\leq m\leq d-1$. The \emph{spinal Hanoi towers group of type $(d,m)$ on $\T_d$} is the group
\[
G^{(d,m)}=\langle a_1,\dots,a_d\rangle\leq \Aut(\T_d),
\]
where
\[
a_i=(1,\stackrel{i-1}{\dots},1,a_i,1,\dots,1)\sigma_i
\quad\text{and}\quad \sigma_i=(i+1\ \ i+2\ \ \cdots\ \ i+m),
\]
for every $i\in\{1,\dots,d\}$, with indices taken modulo $d$.
\end{definition}
Note that $m\geq 2$ is required to avoid the trivial group, while $m\leq d-1$ is forced by the condition that $\sigma_{i}$ must fix $i$.

The original Hanoi towers group of Grigorchuk and \v{S}uni\'c~\cite{Sunic} coincides with $G^{(3,2)}$, while Skipper's groups~\cite{Rachel} correspond to the case of maximum type where $m=d-1$.

Our goal is to show that spinal Hanoi towers groups of type $(d,m)$ can be well understood with the machinery developed in the previous sections. In principle, by \Cref{lemma:J-contains-derived-sections}, it is enough to prove that $J\leq G^{(d,m)}$, since this inclusion implies that $G^{(d,m)}$ is regular branch over its derived subgroup. However, in general it is easier to proceed in the opposite order: we first prove  that $G^{(d,m)}$ is regular branch over $(G^{(d,m)})'$ in \Cref{propstronglyfract}, and then use this branch structure to establish the inclusion $J\leq G^{(d,m)}$.

\begin{lemma}\label{permutationgenerating}
   Let $G^{(d,m)}$ be a spinal Hanoi towers group of type $m$. Then
   $$P=\begin{cases}
\alt(d) & \text{if $m$ is odd,}\\
\Sym(d) & \text{if $m$ is even,} \\
\end{cases}$$
where $\alt(d)$ denotes the alternating group on $d$ elements, and $\Sym(d)$ the symmetric group on $d$ elements.
\end{lemma}
\begin{proof}
	All indices are taken modulo $d$. We first note that $P$ contains $\alt(d)$. Indeed,
	\[
	\sigma_{i+1}^{-1}\sigma_i=(i+m+1\ \ i+1\ \ i+2),
	\]
    and for $i\in\{1,\ldots,d\}$, these  elements generate $\alt(d)$.
	
	It remains only to distinguish the parity. If $m$ is odd, then each $\sigma_i$ is an even permutation, so $P\leq \alt(d)$, and hence $P=\alt(d)$. If $m$ is even, then each $\sigma_i$ is odd. Since $\alt(d)\leq P$ and $P$ contains an odd permutation, we get $P=\Sym(d)$.
\end{proof}

We now record  some computations needed below, distinguishing the cases $d\geq 5$ and $d=4$.

\begin{lemma}\label{lem:sections-m2}
	Let $G=G^{(d,2)}$ with $d\geq 5$. Then, for every $i$,
	\[
	\psi([a_i,a_{i-2}])
	=(1,\ldots,1,
	\underset{i-1}{a_i},
	\underset{i}{a_i^{-1}},1,\ldots,1).
	\]
\end{lemma}

\begin{proof}
	Here $\sigma_i=(i+1\ i+2)$ and $\sigma_{i-2}=(i-1\ i)$ are disjoint, hence commute. Thus
	\[
	\psi(a_i^{a_{i-2}})
	=(1,\ldots,1,\underset{i-1}{a_i},1,\ldots,1)\sigma_i.
	\]
	The result follows by multiplying with $a_i^{-1}$.
\end{proof}

In the following lemma all indices are taken modulo $4$.
\begin{lemma}\label{lem:smallsupport42}
	Let $G=G^{(4,2)}$. Write $x_i=(a_i a_{i+1}a_{i+2}a_{i+1})^2(a_i a_{i-1})^2$. Then
    \[
	\psi(x_i)=(\underset{i}{a_i},1,
	\underset{i+2}{a_i},1).
	\]
\end{lemma}

\begin{proof}
	It is enough to prove the claim for $i=1$, since the recursion is invariant under the cyclic relabeling $a_j\mapsto a_{j+1}$. Set $	u=a_1a_2a_3a_2$, and  $	v=a_1a_4$. Using the wreath-product multiplication rule,
    %and the four defining recursions
     we get
	\[
	\psi(u^2)=(1,a_1,a_1,1)(1\ 3\ 2),
	\qquad
	\psi(v^2)=(1,a_1,a_1,1)(1\ 2\ 3).
	\]
	The root permutations cancel, and another multiplication gives
	\[
	\psi(u^2v^2)=(a_1,1,a_1,1).
	\]
	Since $u^2v^2=x_1$, this proves the claim.
\end{proof}

\begin{lemma}\label{lem:sections-middle}
	Let $G=G^{(d,m)}$ with $d\geq 5$ and $3\leq m\leq d-2$. For every $i$, set $	y_i=[a_{i-1},a_i]^2$. Then
	\[
	\psi(y_i)=
	(1,\ldots,1,
	\underset{i}{a_i},
	\underset{i+1}{a_i^{-1}},
	\underset{i+2}{a_i^{-1}},
	1,\ldots,1,
	\underset{i+m}{a_i},
	1,\ldots,1).
	\]
	Moreover, if $z_i=y_i(y_i^{a_{i-1}})^{-1}$, then $z_i\in\St_G(1)$ and
	\[
	\psi(z_i)=
	\begin{cases}
		(1,\ldots,1,\underset{i}{a_i^{-1}},
		\underset{i+1}{a_i},1,\ldots,1), & m=3,\\[4pt]
		(1,\ldots,1,\underset{i}{a_i},
		\underset{i+1}{a_i^{-2}},1,
		\underset{i+3}{a_i},1,\ldots,1), & m\geq 4.
	\end{cases}
	\]
\end{lemma}

\begin{proof}
	We have
	\[
	\sigma_{i-1}=(i\ i+1\ \cdots\ i+m-1),
	\qquad
	\sigma_i=(i+1\ i+2\ \cdots\ i+m),
	\]
	and hence
	\[
	[\sigma_{i-1},\sigma_i]=(i\ i+m)(i+1\ i+2).
	\]
	A direct computation gives
	\[
	\psi([a_{i-1},a_i])
	=(1,\ldots,1,
	\underset{i+1}{a_i^{-1}},
	1,\ldots,1,
	\underset{i+m}{a_i},
	1,\ldots,1)(i\ i+m)(i+1\ i+2),
	\]
	so squaring gives the displayed formula for $\psi(y_i)$.
	
	On the other hand, conjugation by $a_{i-1}$ introduces no conjugation in the non-trivial sections: the only non-trivial section of $a_{i-1}$ lies at $i-1$, and $i-1\notin\{i,i+1,i+2,i+m\}$ because $m\leq d-2$. Thus $y_i^{a_{i-1}}$ is obtained from $y_i$ by permuting the first level sections by $\sigma_{i-1}$. Multiplying $y_i$ by $(y_i^{a_{i-1}})^{-1}$ its inverse then gives the two displayed forms. For $m=3$ we use $a_i^2=a_i^{-1}$.
\end{proof}

\begin{proposition}\label{propstronglyfract}
	Every group $G=G^{(d,m)}$ is strongly fractal. Moreover, it is regular branch over its derived subgroup.
\end{proposition}

\begin{proof}
 First we prove strong fractality. It is enough to find, for each generator $a_i$, an element of $\St_G(1)$ having $a_i$ as a first level section: by the same minimal-word argument used in the proof of \Cref{fractandtrans}, conjugating by suitable elements of $G$ moves this section to any prescribed coordinate without changing it.
	
	The required elements are as follows.  For $(d,m)=(3,2)$ this is proved in \cite[Example~4]{Sunic}. For $(d,m)=(4,2)$ we use \Cref{lem:smallsupport42}. For $d\geq 5$ and $m=2$ we use \Cref{lem:sections-m2}. For the maximal case $m=d-1$ this is Skipper's result \cite[Theorem~3.16]{Rachel}. Finally, for $d\geq 5$ and $3\leq m\leq d-2$ we use the element $y_i$ from \Cref{lem:sections-middle}. Hence $G$ is strongly fractal.
	
	We now prove regular branch over $G'$. The cases $(d,m)=(3,2)$ and $m=d-1$ are known from \cite{Sunic} and \cite{Rachel}, respectively. In all remaining cases we apply \Cref{strongfractimpliesbranch}. Strong fractality has just been proved, and $\alt(d)\leq P$ by \Cref{permutationgenerating}. It remains to produce, for each $i$, an element of $\St_G(1)$ whose non-trivial sections are powers of $a_i$ and whose support $S_i$ satisfies $2|S_i|-1\leq d$.
	
	For $(d,m)=(4,2)$, the elements of \Cref{lem:smallsupport42} have support of size $2$. For $d\geq 5$ and $m=2$, the elements of \Cref{lem:sections-m2} also have support of size $2$. For $d\geq 5$ and $3\leq m\leq d-2$, we use the elements $z_i$ from \Cref{lem:sections-middle}: their support has size $2$ if $m=3$ and size $3$ if $m\geq 4$. Thus $2|S_i|-1\leq d$ in every case, and the criterion applies.
\end{proof}

\begin{proposition}\label{IcontainedinG'}
	Let $G=G^{(d,m)}$. Then $J\leq G$ if and only if  $d\geq 4$.
\end{proposition}
\begin{proof}
Suppose $d< 4$. Since $d\geq 3$ and $2\leq m\leq d-1$, we necessarily have  $(d,m)=(3,2)$. Hence $G=G^{(3,2)}$ is the classical Hanoi towers group and  $J\not\leq\mathcal{H}^{(3)}$, see \cite[Remark 3.20]{Rachel}.

Now assume $d\geq 4$. By \Cref{permutationgenerating}, the group $P$ is $2$-transitive. By \Cref{propstronglyfract}, the group $G$ is regular branch over $G'$. Therefore, by the transfer criterion \Cref{Transfer}, it is enough to show that, for each $i$, there exists $g\in\St_G(1)$ such that
	\[
	\psi(g)\equiv
	(1,\ldots,1,\underset{p}{a_i},1,\ldots,1,
	\underset{q}{a_i^{-1}},1,\ldots,1)
	\pmod{G'\times\cdots\times G'}
	\]
	for some distinct coordinates $p,q$.
	
	If $(d,m)=(4,2)$, this follows from \Cref{lem:smallsupport42}, since $a_i^{-1}=a_i$. If $d\geq 5$ and $m=2$, it follows from \Cref{lem:sections-m2}. If $m=d-1$, this is precisely \cite[Proposition~3.19]{Rachel}.
	
	It remains to consider $d\geq 5$ and $3\leq m\leq d-2$. Put again $y_i=[a_{i-1},a_i]^2$. From \Cref{lem:sections-middle},
	\[
	\psi(y_i)\equiv
	(1,\ldots,1,
	\underset{i}{a_i},
	\underset{i+1}{a_i^{-1}},
	\underset{i+2}{a_i^{-1}},
	1,\ldots,1,
	\underset{i+m}{a_i},
	1,\ldots,1)
	\pmod{G'\times\cdots\times G'}.
	\]
	Assume first that either $d\geq 7$ or $m$ is even. Then $P$ is at least $5$-transitive. Hence we may conjugate two copies of $y_i$ so that, modulo $G'\times\cdots\times G'$,
	\[
	\psi(y_i^u)\equiv(a_i,a_i,a_i^{-1},a_i^{-1},1,\ldots,1)
	\]
	and
	\[
	\psi(y_i^v)\equiv(a_i^{-1},1,a_i,a_i,1,a_i^{-1},1,\ldots,1).
	\]
	Multiplying these two congruences gives
	\[
	\psi(y_i^u y_i^v)\equiv (1,a_i,1,1,1,a_i^{-1},1,\ldots,1),
	\]
	which is of the required form.
	
	Finally suppose $m=3$ and $d\in\{5,6\}$. Use the ordered coordinates $i-1,i,i+1,i+2,i+3$; when $d=6$, the remaining coordinate is omitted from the display and has entry $1$. Then
	\[
	\psi(y_i)\equiv(1,a_i,a_i^{-1},a_i^{-1},a_i).
	\]
	Since $P=\alt(d)$ contains the $5$-cycle
	\[
	\pi_i=(i-1\ i\ i+1\ i+3\ i+2)
	\]
	(fixing the remaining coordinate when $d=6$), choose $h_i\in G$ with root permutation $\pi_i$. Conjugation by $h_i$ only permutes coordinates modulo $G'\times\cdots\times G'$, and therefore
	\[
	\psi(y_i^{h_i})\equiv(a_i^{-1},1,a_i,a_i,a_i^{-1}).
	\]
	Thus
	\[
	\psi(y_i y_i^{h_i})\equiv(a_i^{-1},a_i,1,1,1),
	\]
	again with all omitted coordinates equal to $1$. This is of the required form, and the proof is complete.
\end{proof}
\begin{remark}
	The derived subgroup need not be the maximal branching subgroup. If $m$ is odd and $d\geq 5$, then $P=\alt(d)$ is perfect, so \Cref{perfectcase} together with \Cref{IcontainedinG'} shows that $G$ is regular branch over itself. The case $d=4$, $m=3$ is exceptional here because $\alt(4)$ is not perfect. Skipper~\cite{Rachel} determines the maximal branching subgroups in the maximal-type case $m=d-1$.
\end{remark}

\begin{theorem}\label{thm:typem-just-infinite}
	Let $G=G^{(d,m)}$ be a spinal Hanoi towers group of type $(d,m)$. Then $G$ is just infinite if and only if $d\geq 4$. Furthermore
    \begin{enumerate}
    \item If $m$ is even and $d \geq 4$, then
	\[
	\psi(\St_G(1))=
	\left\{
	(g_1,\ldots,g_d)\in G\times \stackrel{d}{\cdots}\times G\ \middle|
	\sum_{j=1}^d\sum_{i=1}^d e_i(g_j)\equiv 0 \pmod 2
	\right\}.
	\]
    \item  If $m$ is odd and $d\geq 5$, then
	\[
	\psi(\St_G(1))=G\times\stackrel{d}{\cdots}\times G,
	\]
	equivalently $\psi(G)= (G\times \stackrel{d}{\cdots}\times G)\rtimes\alt(d)$.
    \item If $(d,m)=(4,3)$, then
    \[
	\psi\bigl(\St_G(1)\bigr)=
		\left\{
		(g_1,g_2,g_3,g_4)\in G\times \stackrel{4}{\cdots}\times G  \middle|
		\sum_{j=1}^{4}\sum_{i=1}^{4}(-1)^{i-1}e_i(g_j)\equiv 0 \pmod{3}
		\right\}.
    \]
    \end{enumerate}
\end{theorem}

\begin{proof}
Suppose first that $d\geq 4$. By \Cref{IcontainedinG'}, we have $J\leq G$. Moreover, by \Cref{fractandtrans} and \Cref{permutationgenerating}, the group $G$ is level transitive. Therefore, \Cref{thm:justinfinite} implies that $G$ is just infinite.

Suppose now that $d=3$. Since $2\leq m\leq d-1$, we necessarily have $m=2$. Hence $G=G^{(3,2)}$, which is the classical Hanoi towers group $\mathcal{H}^{(3)}$. This group is not just infinite. Consequently, $G$ is just infinite if and only if $d\geq 4$.

It remains to describe the image of the first-level stabilizer. From now on, assume that $d\geq 4$.

Suppose first that $m$ is even. By \Cref{permutationgenerating}, we have $	P=\Sym(d)$. Each root permutation $\sigma_i$ is an $m$-cycle. Since $m$ is even, every $\sigma_i$ is an odd permutation. Therefore, \Cref{cor:symcase} gives
\[
\psi(\St_G(1))=
\left\{
(g_1,\ldots,g_d)\in G\times \stackrel{d}{\cdots}\times G\ \middle|
\sum_{j=1}^d\sum_{i=1}^d e_i(g_j)\equiv 0 \pmod 2
\right\}.
\]

Suppose next that $m$ is odd and $d\geq 5$. By \Cref{permutationgenerating}, $P=\alt(d)$. Since $\alt(d)$ is perfect for $d\geq 5$, \Cref{perfectcase} yields
\[
	\psi(\St_G(1))=G\times\stackrel{d}{\cdots}\times G,
\]
Since the image of the action of $G$ on the first level is $\alt(d)$, it follows that $\psi(G)= (G\times \stackrel{d}{\cdots}\times G)\rtimes\alt(d)$.

It remains to consider the case where $m$ is odd and $d<5$. Since $m\geq 3$ and $m\leq d-1$, the only remaining possibility is $(d,m)=(4,3)$. In this case $P=\alt(4)$. It is well known that $\alt(4)'$  is the Klein four-group $K_4$ and also that $P/P'=\alt(4)/K_4\cong \mathbb{Z}/3\mathbb{Z}$. Moreover,
\[
\sigma_1K_4=\sigma_3K_4,
\qquad
\sigma_2K_4=\sigma_4K_4=(\sigma_1K_4)^{-1}.
\]
Choosing the isomorphism $P/P'\to\mathbb{Z}/3\mathbb{Z}$ that maps
$\sigma_1K_4$ to $1$, \Cref{cor:rem} gives
\[
\psi\bigl(\St_G(1)\bigr)=
	\left\{
	(g_1,g_2,g_3,g_4)\in G\times \stackrel{4}{\cdots}\times G  \middle|
	\sum_{j=1}^{4}\sum_{i=1}^{4}(-1)^{i-1}e_i(g_j)\equiv 0 \pmod{3}
	\right\}.
\]
\end{proof}

% \fc{The results of this sections conclude the proof of \Cref{thmC}. For $d\geq 3$ and $2\leq m\leq d-1$, strong fractality and regular branchness over the commutator subgroup of $G^{(d,m)}$ are proved in \Cref{propstronglyfract}. For  $d\geq 4$, the containment $J\leq G^{(d,m)}$   is proved in \Cref{IcontainedinG'}, while just infiniteness follows from \Cref{thm:typem-just-infinite}. }

\end{document}